\documentclass[11pt,twoside]{amsart}
\usepackage{amsfonts}
\usepackage{amssymb}
\usepackage{amscd}
\input xy
\xyoption{all}
\title[exceptional representations and Richardson elements]{Exceptional representations of a double quiver of type $A$, and Richardson
elements in seaweed Lie algebras}
\author{Bernt Tore Jensen, Xiuping Su and Rupert W.T. Yu}
\thanks{The first author is supported by NFR project HoGeMetAlg}
\thanks{The second author is supported by Marie-Curie Fellowship IIF}

\newenvironment{acknowledgements}{\section*{Acknowledgments}}{}

\newtheorem{theorem}{Theorem}[section]
\newtheorem{lemma}[theorem]{Lemma}
\newtheorem{proposition}[theorem]{Proposition}
\newtheorem{definition}[theorem]{Definition}
\newtheorem{example}[theorem]{Example}
\newtheorem{question}[theorem]{Question}

\newtheorem{corollary}[theorem]{Corollary}
\newtheorem{remark}[theorem]{Remark}
\newtheorem{remarks}[theorem]{Remarks}

\newcommand{\rep}{\mathrm{Rep}}

\newcommand{\supp}{\mathrm{supp}}
\renewcommand{\hom}{\mathrm{Hom}}
\newcommand{\gl}{\mathrm{GL}}

\newcommand{\ext}{\mathrm{Ext}}

\newcommand{\soc}{\mathrm{Soc}}

\newcommand{\dv}[1]{{\bf #1}}
\newcommand{\minus}{\backslash}
\newcommand{\field}{{k}}
\newcommand{\cok}{\mathrm{Cok}}
\renewcommand{\ker}{\mathrm{Ker}}
\newcommand{\im}{\mathrm{Im}}

\newcommand{\lra}{\longrightarrow}

\newcommand{\ra}{\rightarrow}
\newcommand{\sdp}{\times\kern-.2em\vrule height1.1ex depth-.05ex}
\newcommand{\epi}{\lra \kern-.8em\ra}

\normalbaselines
\begin{document}

\begin{abstract}
In this paper, we study the set of $\Delta$-filtered modules of quasi-hereditary algebras arising
from quotients of the double of quivers of type $A$. Our main result is that for any fixed $\Delta$-dimension
vector, there is a unique (up to isomorphism) exceptional $\Delta$-filtered module. We then apply this
result to show that there is always an open adjoint orbit in the nilpotent radical of a seaweed Lie algebra
in $\mathrm{gl}_{n}(\field )$, thus answering positively in this $\mathrm{gl}_{n}(\field )$ case
to a question raised independently by Michel Duflo and Dmitri Panyushev. An example of a seaweed
Lie algebra in a simple Lie algebra of type $E_{8}$ not admitting an open orbit in its
nilpotent radical is given.
\end{abstract}

\maketitle

\section{Introduction}

Let $\mathfrak{g}$ be a reductive Lie algebra over an algebraically closed field $\field$ of characteristic zero.
A Lie subalgebra $\mathfrak{q}$ of $\mathfrak{g}$ is called a seaweed Lie algebra if there exists a
pair of parabolic subalgebras $(\mathfrak{p},\mathfrak{p}')$ of $\mathfrak{g}$ such that
$\mathfrak{q}=\mathfrak{p}\cap \mathfrak{p}'$ and $\mathfrak{p}+\mathfrak{p}'=\mathfrak{g}$
(we call such a pair of parabolic subalgebras weakly opposite). For example, take the pair
consisting of a Borel subalgebra and its opposite.

Seaweed Lie algebras are introduced by Vladimir Dergachev and Alexandre Kirillov \cite{DK}
in the case
$\mathfrak{g}=\mathrm{gl}_{n}(\field )$,
and in the above
generality by Dmitri Panyushev \cite{P}. The set of seaweed Lie algebras in $\mathfrak{g}$ contains
clearly all parabolic subalgebras of $\mathfrak{g}$ and their Levi factors. In particular, they provide
new examples of index zero Lie algebras (or Frobenius Lie algebras) \cite{DK,P,TYart2}.

A general formula for the index of a seaweed Lie algebra conjectured
in \cite{TYart2} was recently proved by Anthony Joseph
\cite{joseph}. This is an unexpected and pleasant surprise.
Naturally, we would like to know to what extent certain classical
results can be generalized to this large class of Lie subalgebras of
$\mathfrak{g}$.

Let $\mathbf{G}$ be a connected reductive algebraic group whose Lie
algebra is $\mathfrak{g}$. Let $(\mathfrak{p},\mathfrak{p}')$ be a
pair of weakly opposite parabolic subalgebras of $\mathfrak{g}$, and
$\mathfrak{q}=\mathfrak{p}\cap \mathfrak{p}'$ the corresponding
seaweed Lie algebra. Denote by  $\mathbf{P}$ and $\mathbf{P}'$ the
parabolic subgroups of $\mathbf{G}$ corresponding to $\mathfrak{p}$ and
$\mathfrak{p}'$. Set $\mathbf{Q}=\mathbf{P}\cap \mathbf{P}'$.

We are interested in the following question raised by Michel Duflo and Dmitri Panyushev
independently :

\begin{question}\label{mainquestion}
Is there an open $\mathbf{Q}$-orbit in the nilpotent radical
$\mathfrak{n}$ of $\mathfrak{q}$ ?
\end{question}

Equivalently, we may ask if there is an element $x\in \mathfrak{n}$ verifying
$[x,\mathfrak{q}] = \mathfrak{n}$.

In the case where $\mathfrak{q}$ is a parabolic subalgebra of
$\mathfrak{g}$, then the answer is yes, and this is a result that is
commonly known as Richardson's Dense Orbit Theorem
\cite{Richardson}. If there is an open $\mathbf{Q}$-orbit in the
nilpotent radical of $\mathfrak{q}$, then an element in the open
$\mathbf{Q}$-orbit is called a Richardson element of $\mathfrak{q}$.

Using some computations with the program {\sc Gap4}, we are able to check that
Richardson elements exist in any seaweed Lie algebra in a simple Lie
algebra of rank $\leq 7$. However, in type $E_{8}$, we found a seaweed Lie algebra
whose nilpotent  radical does not contain an open orbit.

Our present task is to study the case of seaweed Lie algebras in $\mathrm{gl}_{n}(\field )$.
We prove that :

\begin{theorem}
Any seaweed Lie algebra in $\mathrm{gl}_{n}(\field )$ has a Richardson element.
\end{theorem}

In this particular case, seaweed Lie algebras can be viewed as the stabiliser of
a pair of weakly opposite flags. This provides a nice
description of seaweed Lie algebras, and allows us to transfer the problem to a
quiver representations setting, extending the one for parabolic subalgebras considered by
Thomas Br\"{u}stle,  Lutz Hille, Claus Ringel and  Gerhard R\"{o}hrle
\cite{BHRR}.

More precisely, we associate to $\mathfrak{q}$ a certain double
quiver $\tilde{Q}$ of type $A$ together with a dimension vector $\dv
d$. Then the quotient of the path algebra of $\tilde{Q}$ by certain
relations $\mathcal{I}$ has a structure of a quasi-hereditary
algebra with respect to some partial order on the vertices. See
Section \ref{doublequiver} for definitions and basic properties for
quasi-hereditary algebras.

A family of quasi-hereditary algebras constructed as quotients of the double of quivers has
appeared in different contexts, see \cite{D,DX,HV,KX}. These
quasi-hereditary algebras are sometimes called the twisted double incidence algebras of
posets, see for example \cite{DX}.

The existence of an open $\mathbf{Q}$-orbit in the nilpotent radical
of $\mathfrak{q}$ corresponds then to the existence of an open $\gl
(\dv d)$-orbit in the set  $\rep_{\Delta}(\tilde{Q} ,
\mathcal{I},\dv d)$ of $\Delta$-filtered modules of dimension vector
$\dv d$ verifying $\mathcal{I}$, where $\dv d$ is a dimension vector
completely determined by the pair of weakly opposite flags
associated to $\mathfrak{q}$.

The main theorem we prove is:

\begin{theorem}
There exists a unique (up to isomorphism) exceptional
$\Delta$-filtered module for any given $\Delta$-dimension vector.
\end{theorem}

The idea of the proof consists firstly of constructing exceptional $\Delta$-filtered
modules with linear $\Delta$-support, and then of constructing
an exceptional $\Delta$-filtered module by gluing together in a specific way
these exceptional $\Delta$-filtered modules with linear $\Delta$-support.
This explicit construction has the advantage of providing an explicit Richardson element.
These constructions extend those used in \cite{BHRR}.

The paper is organized as follows. In Sections \ref{doublequiver} and \ref{doublequiverA},
we give some generalities of quasi-hereditary algebras arising from quotients of the
double of quivers. In particular, we show that any $\gl (\dv d)$-orbit in
$\rep_{\Delta}(\tilde{Q} , \mathcal{I},\dv d)$ meets a certain vector space $R^{\alpha}$.
In Section \ref{gluing:section}, we prove some properties of the gluing process.
We generalize in Section \ref{exceptionallinear:section}
certain results in \cite{BHRR} to exceptional $\Delta$-filtered modules
of linear $\Delta$-support.  We prove our main theorem on the existence of an open
$\gl (\dv d)$-orbit in $\rep_{\Delta}(\tilde{Q} , \mathcal{I},\dv d)$ in Section \ref{exceptional:section}.
After reviewing some basic properties of seaweed Lie algebras in Section \ref{seaweed:section},
we examine the case of seaweed Lie algebras in $\mathrm{gl}_{n}(\field )$
in Section \ref{typeA:section} where we explain how to associate to a seaweed Lie algebra
the double of a quiver of type $A$, and the correspondence between orbits. Section
\ref{others:section} is devoted to seaweed Lie algebras in other simple Lie algebras where
we give an example of a seaweed Lie algebra not admitting a
Richardson element.

\section{Quasi-hereditary algebras arising from quotients of the double of quivers}
\label{doublequiver}

In this section, we first introduce necessary notation on
representations of quivers. Then we recall the construction
of a class of quasi-hereditary algebras as quotients of the double of quivers.  We
also prove some preliminary results on this class of quasi-hereditary algebras.
Note that one of the important ingredients for
quasi-hereditary algebras is $\Delta$-filtered modules.
At the end of this section, we describe
varieties of $\Delta$-filtered modules of the quasi-hereditary algebras.

\subsection{Representations of quivers}

Let $Q=(Q_0, Q_1, s, t)$ be a quiver, where $Q_0$ is the set of
vertices, $Q_1$ is the set of arrows and $s$ and $t$ are two maps
from $Q_1$ to $Q_0$ given by sending each arrow in $Q_1$ to its
starting vertex and its terminating vertex, respectively. We assume
that both $Q_0$ and $Q_1$ are finite sets. A vertex $i\in Q_0$ is
called a sink vertex if there are no arrows in $Q_1$ starting from
$i$ and a source vertex if there are no arrows in $Q_1$ terminating
at $i$. We say that $i$ is an admissible vertex if it is either a
source or a sink vertex. We denote by $M=(\{M_i\}_{i\in
Q_0},\{M_\alpha\}_{\alpha\in Q_1} )$ a representation of $Q$, where
each $M_i$ is a $k$-vector space and $M_\alpha$ is a $k$-linear map
from $M_{s(\alpha)}$ to $M_{t(\alpha)}$. We view each vertex as a
trivial path in $Q$. A non-trivial path in $Q$ is a sequence of
arrows $\rho_1\cdots\rho_l$ satisfying $t(\rho_i)=s(\rho_{i-1})$ for
any $i\geq 2$. We denote by $kQ$ the path algebra of $Q$.  The path
algebra $kQ$ is a vector space with all the paths of $Q$ as basis
and  the multiplication of any two paths $\rho$ and $\gamma$ given
by

$$\rho\cdot \gamma = \left\{\begin{array}{ll} \rho\gamma & \mbox{if } t(\gamma)=s(\rho);\\
 0 &\mbox{ otherwise}. \end{array}\right.$$

Let $\mathcal{J}$ be an ideal of the path algebra $kQ$.
A representation $M=(\{M_i\}_{i\in
Q_0},\{M_\alpha\}_{\alpha\in Q_1} )$ of $Q$ is a representation of
$(Q, {\mathcal J})$ if the maps $M_\alpha$ satisfy the relations in
$\mathcal{J}$, that is, $\sum_{i_1, \cdots, i_l}c_{i_1, \cdots, i_l}
M_{\alpha_{i_1}}\cdots M_{\alpha_{i_l}}=0$ if $\sum_{i_1, \cdots,
i_l}c_{i_1, \cdots, i_l}{\alpha_{i_1}}\cdots {\alpha_{i_l}}\in
\mathcal{J}$. A morphism of two representations $N$ and $M$ of $(Q,
{\mathcal J})$ is a family of linear maps $(f_i: N_i\lra M_i)_{i\in
Q_0}$ such that $M_\alpha f_{s(\alpha)}= f_{t(\alpha)}N_\alpha$ for
any $\alpha\in Q_1$. It is well-known that the category of
representations of $(Q, {\mathcal J})$ is equivalent to the category
of left modules of $kQ/{\mathcal J}$.
We do not distinguish a representation of
$(Q, {\mathcal J})$ from the corresponding $ kQ/{\mathcal J}$-module.

Another way to understand representations of a quiver is through
the representation variety $\rep(Q, \mathcal{J}, \dv d)$
with dimension vector $\dv d=(d_i)_{i\in Q_0}$ defined as follows.

$$\rep(Q, {\mathcal J}, \dv d)=\left\{(M_\alpha)_{\alpha \in Q_1}\left|
\begin{array}{l}
\mbox{ each entry }M_\alpha \in
\hom(k^{d_{s(\alpha)}}, k^{d_{t(\alpha)}}) \\
\mbox{ and } (M_\alpha)_{\alpha \in Q_1} \mbox{ satisfies relations in } {\mathcal J}
\end{array}\right. \right\}.$$


Note that $\rep(Q, {\mathcal J}, \dv d)$ is a closed subvariety of
the affine space $\prod_{\alpha\in {Q}_1} \hom(k^{d_{s(\alpha)}},
k^{d_{t(\alpha)}})$. In the case where ${\mathcal J}=0$, we have
$\rep(Q, {\mathcal J}, \dv d)= \prod_{\alpha\in {Q}_1}
\hom(k^{d_{s(\alpha)}}, k^{d_{t(\alpha)}})$.  The group $\gl(\dv
d)=\prod_i \gl_{d_i}(k)$ acts on $\rep(Q, {\mathcal J}, \dv d)$ by
conjugation and there is a one-to-one correspondence between the
$\gl(\dv d)$-orbits in $\rep(Q, {\mathcal J}, \dv d)$  and
isomorphism classes of $kQ/{\mathcal J}$-modules with dimension
vector $\dv d$.

\subsection{Quasi-hereditary algebras and the double  of quivers}\label{q-hered}
We shall recall the definition of quasi-hereditary algebras, see \cite{CPS, KX}.

\begin{definition}

Let $D$ be a finite-dimensional $\field$-algebra with a set of
representatives of simple $D$-modules $\{L(i)\}_{i\in I}$. Then $D$, together with a partial
order $(I,\succeq)$, and a set of modules $\{\Delta(i)\}_{i\in I}$, is
called a quasi-hereditary algebra if,
\begin{itemize}

\item[(1)] each $\Delta(i)$ has simple top $L(i)$,

\item[(2)] we have $j\preceq i$ for any  composition factor $L(j)$ of the radical of $\Delta(i)$,

\item[(3)] the kernel of the projective cover $P(i)\lra \Delta(i)$ is
  filtered by the $\Delta(j)$ with $j\succeq i$.

\end{itemize}

The modules in $\Delta = \{\Delta(i)\}_{i\in I}$ are called Verma
modules. The subcategory of $D$-modules which are filtered by Verma
modules is denoted by $\mathcal{F}(\Delta)$. We say that $M$ is
$\Delta$-filtered if $M\in \mathcal{F}(\Delta)$.

\end{definition}

\begin{definition} A $\Delta$-filtration of a $\Delta$-filtered module $M$ is defined
to be a descending chain
of submodules, $M=M(1)\supset M(2)\supset \cdots \supset M(r)=0$,
such that  for any $i$, the module $M(i)/M(i+1)$ is isomorphic to $\Delta(j)$ for some $j$.
Each $\Delta(l)$ appearing as a
direct summand of $M(i)/M(i+1)$ for some $i$ is called a $\Delta$-composition
factor of $M$. The $\Delta$-length of $M$ is defined to be the number of
$\Delta$-composition factors of $M$.
\end{definition}

From now on, we assume that the quiver $Q$ has no oriented cycles.
We denote by $\tilde{Q}$ the double of $Q$,
that is $\tilde{Q}_0=Q_0$ and $\tilde{Q}_1=Q_1\cup \{\alpha^*|\alpha^* \mbox{ is a reverse arrow of } \alpha\in Q_1\}$. Denote
by $k\tilde{Q}$ the path algebra of $\tilde{Q}$ . Let $\mathcal{I}$ be the ideal
of $k\tilde{Q}$ generated by the following elements

\begin{itemize}
\item[(1)] $\alpha^*\alpha-\sum_{\gamma\in Q_1, s(\alpha)=t(\gamma)}\gamma\gamma^*$
for any $\alpha\in Q_1$.

\item[(2)] $\beta^*\alpha$, where $\beta\not= \alpha\in Q_1$ with $t(\alpha)=t(\beta)$.
\end{itemize}

Note that the inclusion of quivers $Q\subseteq \tilde{Q}$, makes $kQ$ a
subalgebra of $k\tilde{Q}/\mathcal{I}$.
Also, mapping each $\alpha^*\in \tilde{Q}_1\minus Q_1$ to zero gives us a
surjective map of algebras $k\tilde{Q}/\mathcal{I}\lra kQ$.
We view any $k\tilde{Q}/\mathcal{I}$-module as an $kQ$-module and
any $kQ$-module as a $k \tilde{Q}/\mathcal{I}$-module via these two algebra homomorphisms.

Let $\Delta(i)$ be an indecomposable projective $kQ$-module with
simple top $L(i)$. We define a binary relation $\succeq$ on the vertex set $Q_0$ by
$i\succeq j$ if $\Delta(j)\subseteq \Delta(i)$ as $kQ$-modules. Since $kQ$ is
finite-dimensional, we see that $\succeq$ is a partial order on $Q_0$.
We let $P(i)$ be a projective $k\tilde{Q}/\mathcal{I}$-module with simple top $L(i)$.

\begin{proposition}\label{HVTheorem1}\cite{DX}
The modules $\Delta(i)$ and the partial order $\succeq$ give $k\tilde{Q}/\mathcal{I}$ the
structure of a quasi-hereditary algebra. Moreover $gldim \;k\tilde{Q}/\mathcal{I} \leq 2$.
\end{proposition}

The following result is easy and well-known to experts.

\begin{proposition} \label{HVTheorem2}
The following are equivalent for a representation
$M$
of $(\tilde{Q},\mathcal{I})$.
\begin{itemize}
\item[(i)] The representation $M$ is $\Delta$-filtered.
\item[(ii)] The projective dimension of $M$ is at most one.
\item[(iii)] The representation $M$ is projective as a $kQ$-module.
\item[(iv)] For any arrow $\alpha\in Q_1$, the $\field$-linear map $M_{\alpha}$ is injective and
 $$\im(M_{\alpha})\bigcap \sum_{\begin{smallmatrix}\beta\in Q_1\minus\{\alpha\}\\ t(\beta)=t(\alpha)\end{smallmatrix}}\im(M_{\beta})=0.$$
\end{itemize}
\end{proposition}
A proof of the equivalence of (i)-(iii) of Proposition
\ref{HVTheorem2} can be found in \cite{HV} and the equivalence of
(iii) and (iv) is clear for any quiver.

\begin{definition}
Let $M$ be a $\Delta$-filtered module.
\begin{itemize}
\item[(1)] The
$\Delta$-dimension vector of $M$, denoted by
$\underline{\dim}_\Delta(M)$,  is the dimension vector with its i-th
entry $(\underline{\dim}_\Delta(M))_i$ the multiplicity of $\Delta(i)$
as a $\Delta$-composition factor in a $\Delta$-filtration of $M$.
\item[(2)]The $\Delta$-support of $M$, denote by $\supp_\Delta(M)$, is the
the full subgraph of $Q$ with the set of
vertices $\{i\in Q_0|  (\underline{\dim}_\Delta(M))_i>0\}$.
\end{itemize}
\end{definition}

Since any $\Delta$-filtered module $M$ is projective as a
$kQ$-module, there is a unique decomposition, up to isomorphism, of
$M$ as a $kQ$-module into a direct sum of projective $kQ$-modules.
So the  $\Delta$-dimension vector of a $\Delta$-filtered module is
well-defined. We will also need the usual support of $M$, denoted by
$\supp(M)$, defined via the ordinary dimension vector
$\underline{\dim}(M)$. We have $\supp_\Delta(M)\subseteq \supp(M)$
with strict inclusion in general. Given a non-zero vector $\dv d \in
\mathbb{N}^{Q_{0}}$, we denote by $\supp(\dv d)$ the support of $\dv d$,
which is the full subquiver of $Q$ with the set of vertices
$\supp(\dv d)_0=\{i|d_i>0\}$.

We need some properties of the $\Delta$-filtered modules.
By the definition of  $\mathcal{I}$, it is not difficult to see that the
$k\tilde{Q}/\mathcal{I}$-module $\Delta(i)$ is projective for a source $i$,
that is $\ext^1_{k\tilde{Q}/\mathcal{I}}(\Delta(i),-)=0$. A similar, but weaker,
property holds for sinks.

\begin{lemma} \label{extinjective}
If $M$ is $\Delta$-filtered and $i$ is a sink in $Q$, then $\ext^1_{k\tilde{Q}/\mathcal{I}}(M,\Delta(i))=0$.
\end{lemma}
\begin{proof}
Observe that $\hom_{k\tilde{Q}/\mathcal{I}}(\Omega\Delta(j),\Delta(i))=0$ for all $j$, where
$\Omega\Delta(j)$ denotes the syzygy of $\Delta(j)$ as a $k\tilde{Q}/\mathcal{I}$-module. Hence
$\ext^1_{k\tilde{Q}/\mathcal{I}}(\Delta(j),\Delta(i))=0$. The lemma follows by using long exact
sequences in homology and induction on the length of
a $\Delta$-filtration of $M$.
\end{proof}
\begin{lemma} \label{subsource}
Let $i$ be a source in $Q$ and let $M$ be $\Delta$-filtered with
$d_i=(\underline{\dim}_\Delta(M))_i>0$, then there exists a unique submodule
$\Delta(i)^{d_i}\subseteq M$. Moreover the quotient module $M/\Delta(i)^{d_i}$ is
$\Delta$-filtered.
\end{lemma}
\begin{proof}
We use induction on the length of a $\Delta$-filtration $$0
\subseteq M_1 \subseteq \dots \subseteq M_t=M$$ of $M$. If $t=1$,
then $M=\Delta(i)$, and we are done. If $t>1$, we have a short exact
sequence 
$$
\begin{CD}
0 @>>> M_{t-1} @>>> M @>>> M/M_{t-1} @>>> 0
\end{CD}
$$ 
If
$M/M_{t-1}\cong \Delta(i)$, then $M\cong M_{t-1}\oplus \Delta(i)$,
since $\Delta(i)$ is projective, and by induction we are done.
Otherwise, by induction we have the pushout $Y$ in the following diagram.

$$
\xymatrix{& 0 \ar[d] & 0 \ar[d] & & \\ & \Delta(i)^{d_i} \ar[d]^{g} \ar@{=}[r] &
  \Delta(i)^{d_i} \ar[d] & & \\ 0 \ar[r] & M_{t-1} \ar[d] \ar[r] & M
  \ar[r] \ar[d]
  & M/M_{t-1} \ar@{=}[d] \ar[r] & 0 \\ 0 \ar[r] & \cok(g) \ar[d] \ar[r] &
  Y \ar[d] \ar[r] &
  M/M_{t-1} \ar[r] & 0 \\ & 0 & 0  }
$$

\noindent Here $\cok(g)$ is $\Delta$-filtered. This shows that $Y$ is
$\Delta$-filtered. This proves the lemma.
\end{proof}



\begin{lemma} \label{factorsink}
Let $i$ be a sink in $Q$ and let $M$ be $\Delta$-filtered with
$d_i=(\underline{\dim}_\Delta(M))_i>0$, then there exists a unique quotient module which is
isomorphic to $\Delta(i)^{d_i}$. Moreover the submodule $Y$, satisfying
 $M/Y\cong \Delta(i)^{d_i}$, is
$\Delta$-filtered.
\end{lemma}
\begin{proof}
We use induction on the length of a $\Delta$-filtration $$0
\subseteq M_1 \subseteq\dots \subseteq M_t=M$$ of $M$. If $t=1$,
then $M=\Delta(i)$ and we are done. If $t>1$, then we have a short
exact sequence
$$
\begin{CD}
0 @>>>  M_1 @>>> M @>>> M/M_1 @>>> 0
\end{CD}
$$ 
If $M_1=\Delta(i)$, then
$M\cong M/M_1 \oplus \Delta(i)$ by Lemma \ref{extinjective}. If not,
then by induction we get the pullback $Y$ in the following diagram.

    $$\xymatrix{& &
   0 \ar[d] & 0 \ar[d] & \\ 0 \ar[r] & M_1 \ar@{=}[d]
    \ar[r] & Y\ar[d] \ar[r] & \ker(g) \ar[d]
    \ar[r] & 0 \\ 0 \ar[r] & M_1 \ar[r] & M \ar[r] \ar[d] & M/M_1 \ar[d]^{g}
    \ar[r] & 0 \\
   && \Delta(i)^{d_i} \ar@{=}[r] \ar[d] & \Delta(i)^{d_i} \ar[d] \\ && 0
& 0}$$
Thus  $Y$ is $\Delta$-filtered, since $M_1$ and $\ker(g)$ are
$\Delta$-filtered. This proves the lemma.
\end{proof}

Given a $\Delta$-filtered module $M$ with $\Delta$-dimension vector
$\dv d=(d_i)_i$, by iteration of Lemma  \ref{subsource} or Lemma
\ref{factorsink} we obtain a nice  $\Delta$-filtration $M=M(0)
\supset M(1)\supset M(2)\supset \cdots $ such that $M(i)/M(i+1)\cong
\Delta(l_i)$. Moreover, for any $i$ either $l_{i+1}\succeq l_{i}$ or
they are non-comparable.

\subsection{Representation variety of $\Delta$-filtered modules}

Let $\rep_\Delta(\tilde{Q}, {\mathcal I}, \dv d)$ be the subset of
$\rep(\tilde{Q}, {\mathcal I}, \dv d)$, containing the
points which correspond to $\Delta$-filtered modules, where $\mathcal{I}$ is the ideal
defined in  Section \ref{q-hered}.
Suppose that $ \rep_\Delta(\tilde{Q}, {\mathcal I}, \dv d)$ is not empty. Then by Proposition \ref{HVTheorem2},
for any vertex $i_0\in Q_0$, we have $d_{i_0}\geq\sum_{ \alpha\in Q_1, t(\alpha)=i_0}d_{s(\alpha)}$.
We let $d'_{i_0}=d_{i_0}-\sum_{ \alpha\in Q_1, t(\alpha)=i_0}d_{s(\alpha)} $ and
 decompose

$$k^{d_{i_0}}= \left(\bigoplus_{\begin{smallmatrix}\alpha\in Q_1, &t(\alpha)=i_0\end{smallmatrix}}
k^{d_{s(\alpha)}}\right) \oplus k^{d'_{i_0}}. $$
Now define

$$R^\alpha=\{M\in  \rep_\Delta(\tilde{Q}, {\mathcal I}, \dv d)| \mbox{ each map }M_\alpha
\mbox{ is the standard embedding } k^{d_{s(\alpha)}}\subseteq
k^{d_{i_0}} \}.$$

\begin{proposition} \label{irreducibility}
Suppose that the set $\rep_\Delta(\tilde{Q}, {\mathcal I}, \dv d)$ is non-empty.
\begin{itemize}
\item[(1)] The subset $R^\alpha$ is an affine space.

\item[(2)] The subset  $\rep_\Delta(\tilde{Q}, {\mathcal I}, \dv d)$ is open and irreducible in
$\rep(\tilde{Q}, {\mathcal I}, \dv d)$.

\end{itemize}
\end{proposition}

\begin{proof}
By the definition of $\mathcal{I}$, we see that any element in
$\mathcal{I}$ is a linear combination of arrows in
$\tilde{Q}_1\minus Q_1$ when we view arrows in $Q_1$ as constants.
Thus $R^\alpha$ is the solution space of a linear system, and so it
is an affine space. The openness of $\rep_\Delta(\tilde{Q}, {\mathcal
I}, \dv d)$ follows from Proposition \ref{HVTheorem2} (iv).  By
Proposition \ref{HVTheorem2} (iii), we know that the $\gl(\dv
d)$-orbit of any $\Delta$-filtered representation $M$ meets with
$R^\alpha$. So we have a surjective map $\gl(\dv d)\times
R^\alpha\rightarrow \rep_\Delta(\tilde{Q}, {\mathcal I}, \dv d)$,
and thus $\rep_\Delta(\tilde{Q}, {\mathcal I}, \dv d)$ is
irreducible.
\end{proof}

\begin{remark}
Note that $\Delta$-filtered modules are characterized by vanishing of
some extension groups (see \cite{DX2}). Therefore, in general, for any quasi-hereditary algebra $\Lambda$, we have
that the module variety $\mathrm{mod}_\Delta(\Lambda, \dv d)$ of $\Delta$-filtered modules with dimension vector
$\dv d$ is open in the whole module variety $\mathrm{mod}(\Lambda, \dv d)$.
\end{remark}

\section{The double of a quiver of type $A_n$}
\label{doublequiverA}

From now on, we shall concentrate on the study of $k\tilde{Q}/\mathcal{I}$, where
$Q$ is a quiver of type $A_n$, and we denote by ${\bf D}$ the algebra $k\tilde{Q}/\mathcal{I}$.
In this section, we first fix some notation concerning
$Q$ of type $A_n$. Then we give an example of ${\bf D}$ with $Q$ of
type $A_5$, illustrating the structure of ${\bf D}$.

Let $Q_0=\{1,\dots,n\}$ be the set of vertices of $Q$.
For $1\leq i\leq n-1$, there is a unique arrow connecting vertices $i$ and $i+1$ and
we denote it by $\alpha_i$.
We denote by $\beta_i$ the reverse arrow of $\alpha_i$ in
$\tilde{Q}_1$, that is, $s(\beta_i)=t(\alpha_i)$ and $t(\beta_i)=s(\alpha_i)$. For the sake
of convenience, we set $\alpha_0, \beta_0, \alpha_n$ and $\beta_n$ to be zero arrows. Now the relations
in $\mathcal{I}$ defined in Section \ref{q-hered} are as follows.


\begin{itemize}
\item[(1)] $\beta_{j-1}\alpha_{j-1}$ and $\beta_j\alpha_j$ for a
source $j$;
\item[(2)]  $\beta_{j}\alpha_{j-1}$ and
$\beta_{j-1}\alpha_j$ for a sink $j$;
\item[(3)]  $\beta_{j-1}\alpha_{j-1}-\alpha_j\beta_j$ for $j$
 non-admissible and $t(\alpha_j)=j$;
\item[(4)] $\alpha_{j-1}\beta_{j-1}-\beta_j\alpha_j$ for $j$
non-admissible and $s(\alpha_j)=j$.
\end{itemize}


We say that $j$ is a successor of $i$, if $i\succ j$ and $i\succeq l\succeq j$
implies $j=l$ or $l=i$. In this case, $i$ is called a predecessor of
$j$. Note that if $i$ is non-admissible then it has a unique
successor and a unique predecessor. If $i$ is sink it has no
successor, and at most two predecessors, and exactly two if and only
if $i$ is an interior vertex of $Q$. Similarly, if $i$ is a source,
then it has no predecessor, and at most two successors, and exactly
two if and only if $i$ is an interior vertex of $Q$. Another
description of $\succeq$ is given by $i\succeq j$ if and only if there is
a path in $Q$ from $i$ to $j$.

\begin{example}\label{example1}
Let $Q$ be the quiver:
$$ \xymatrix{1 &\ar[l] ^{\alpha_1}2\ar[r]_{\alpha_2}&
3\ar[r]_{\alpha_3}&4
&\ar[l]^{\alpha_4}5}
$$

\noindent Then $\tilde{Q}$ is
$$
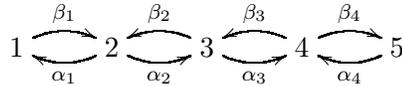
\xymatrix{1\ar@/^/[r]^{\beta_1}&\ar@/^/[l]^{\alpha_1}2\ar@/_/[r]_{\alpha_2}&
3\ar@/_/[l]_{\beta_2}\ar@/_/[r]_{\alpha_3}&4\ar@/_/[l]_{\beta_3}\ar@/^/[r]^{\beta_4}
&\ar@/^/[l]^{\alpha_4}5}$$

\noindent and ${\bf D}=\field\tilde{Q}/\mathcal{I}$ with
$\mathcal{I}$ generated by $\{ \beta_1\alpha_1, \beta_2\alpha_2,
\alpha_2\beta_2-\beta_3\alpha_3, \beta_3\alpha_4, \beta_4\alpha_3,
\beta_4\alpha_4\}$. We denote by $e_i$ the trivial path in
$\tilde{Q}$ corresponding to vertex $i$.

(1) Each $\Delta(i)$ is an indecomposable projective $\field
Q$-module and so it has a $\field$-basis consisting of paths in $Q$
which start from vertex $i$. So $\Delta(1), \dots, \Delta(5)$ are,
respectively, as follows.

$$
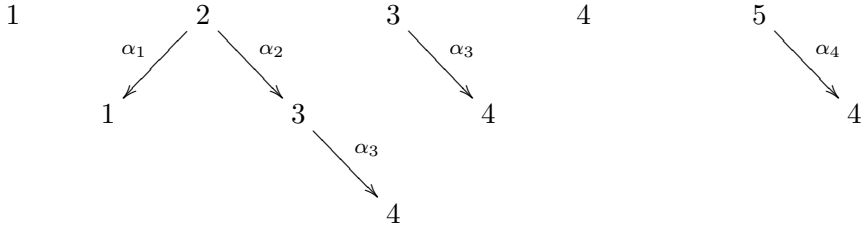
\xymatrix{ 1  && 2 \ar[dl]_{\alpha_1} \ar[dr]^{\alpha_2} && 3 \ar[dr]^{\alpha_3} && 4 && 5 \ar[dr]^{\alpha_4}
  \\ &1& &3\ar[dr]^{\alpha_3} & &4& && & 4  \\ && && 4 &&
  && && }$$

More precisely, the modules $\Delta(1), \dots, \Delta(5)$
have $\field$-bases $\{e_1\}$, $\{e_2, \alpha_1,
\alpha_2,\alpha_3\alpha_2\}$, $\{e_3, \alpha_3\}$, $\{e_4\}$ and
$\{e_5, \alpha_4\}$, respectively.

(2) Now the partial order defined in Section \ref{q-hered} is: $2\succ 1$, $2\succ 3\succ 4$ and $5\succ 4$.


(3) We now compute indecomposable projective ${\bf D}$-modules $P(1), \dots, P(5)$. Each of them has a
$\field$-basis consisting of residue classes of paths
$\overline{p}$, such that $p=ab$, where $a$ consists entirely of
$\alpha_j$'s and $b$ consists entirely of $\beta_j$'s.  More precisely, the modules $P(1), \dots, P(5)$ have
a $k$-basis consisting of $\{e_1, \beta_1,\alpha_1\beta_1, \alpha_2\beta_1, \alpha_3\alpha_2\beta_1 \}$,
$\{e_2, \alpha_1, \alpha_2, \alpha_3\alpha_2\}$, $\{e_3, \beta_2, \alpha_3, \alpha_1\beta_2,$\\$ \alpha_2\beta_2,
\alpha_3\alpha_2\beta_2 \}$,
$\{e_4, \beta_3, \beta_2\beta_3, \alpha_3\beta_3, \alpha_1\beta_2\beta_3, \alpha_2\beta_2\beta_3,
\alpha_3\alpha_2\beta_2\beta_3,  \beta_4, \alpha_4\beta_4 \}$ and $\{e_5, \alpha_4\}$, respectively.

(4) We illustrate the structure of modules $P(1), \dots, P(5)$,
respectively, as follows, where the action by the $\beta_i$'s is
represented by double arrows.

$$
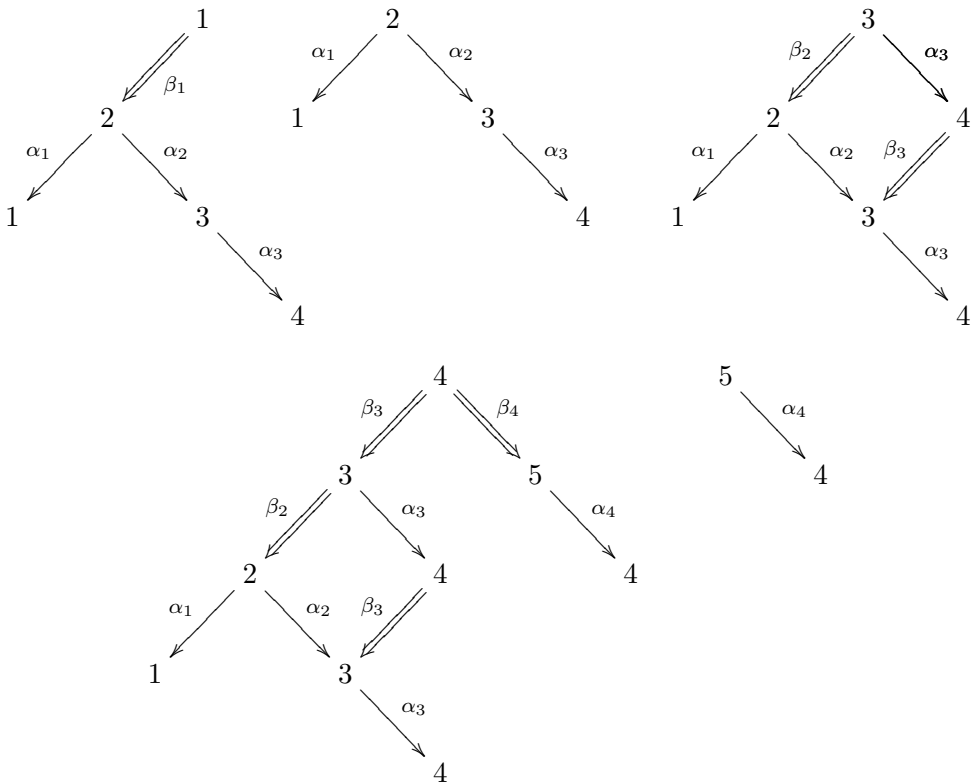
\xymatrix{ &&1\ar@{=>}[dl]^{\beta_1}
  &&2\ar[dr]^{\alpha_2} \ar[dl]_{\alpha_1} &&&&&3\ar[dr]^{\alpha_3}
  \ar@{=>}[dl]_{\beta_2} \ar[dr]^{\alpha_3} \\ &2\ar[dr]^{\alpha_2}
  \ar[dl]_{\alpha_1} &&1&&3\ar[dr]^{\alpha_3}&&&2\ar[dr]^{\alpha_2}
  \ar[dl]_{\alpha_1}&&4 \ar@{=>}[dl]_{\beta_3} \\ 1 &&3\ar[dr]^{\alpha_3}&&&&4&1&&3\ar[dr]^{\alpha_3} \\
  &&&4&&&&&&&4 }$$

$$\xymatrix{ &&& 4\ar@{=>}[dl]_{\beta_3}\ar@{=>}[dr]^{\beta_4} &&& 5\ar[dr]^{\alpha_4} \\ && 3\ar[dr]^{\alpha_3}
  \ar@{=>}[dl]_{\beta_2} &&5\ar[dr]^{\alpha_4} &&& 4 \\ &
  2\ar[dr]^{\alpha_2} \ar[dl]_{\alpha_1} && 4 \ar@{=>}[dl]_{\beta_3} & & 4\\ 1&&3\ar[dr]^{\alpha_3} \\ &&& 4 }$$


(5) For any vertex $i$, the indecomposable projective  ${\bf
D}$-module $P(i)$ has a simple top $L(i)$ and has $\bigoplus_{j}
P(j) $ as its maximal submodule, where $j$'s are predecessors of
$i$. That is, we have a short exact sequence
$$
\begin{CD}
0 @>>> \bigoplus_{j} P(j) @>>> P(i)@>>>
L(i)@>>> 0.
\end{CD}
$$

\end{example}

\begin{remark}\label{basisremark} The properties of ${\bf D}$ in Example \ref{example1}
 can be generalized to $k\tilde{Q}/\mathcal{I}$
for $Q$ of other quivers without oriented cycles, see \cite{HV} for
details.
\end{remark}

\section{Gluing of $\Delta$-filtered modules}\label{gluing:section}

In this section, we prove some further properties of
$\Delta$-filtered ${\bf D}$-modules. We first explain how to glue
two $\Delta$-filtered modules at an admissible vertex and obtain a
new $\Delta$-filtered module, which usually has higher dimension. We
prove that any $\Delta$-filtered module can be obtained by gluing
$\Delta$-filtered modules. Recall that a ${\bf D}$-module is said to
be exceptional if $\ext_{\bf D}^1(M, M)=0$. Suppose that a
$\Delta$-filtered module $M$  is exceptional and is glued from $M'$
and $M''$ at an admissible vertex. Then  we show that both $M'$ and
$M''$ are exceptional.

Let $i$ be a sink or a source of $Q$. Let $M'$ and $M''$ be two
$\Delta$-filtered modules with  $\supp_{\Delta}(M')\subseteq$
$\{1,\dots,i\}$ and  $\supp_{\Delta}(M'')\subseteq$ $\{i,\dots,n\}$,
and $(\underline{\dim}_\Delta(M'))_i=(\underline{\dim}_\Delta(M''))_i=d_i>0$.

Assume that $i$ is a sink in $Q$. By Lemma \ref{factorsink}, we have
short exact sequences 
$$
\begin{CD}
0 @>>> \ker (f') @>>> M' @>{f'}>> \Delta (i)^{d_i} @>>> 0
\end{CD}
$$ 
and 
$$
\begin{CD}
0 @>>> \ker (f'') @>>> M'' @>{f''}>> \Delta (i)^{d_i} @>>> 0.
\end{CD}
$$
Let $M$ be given by
the pullback of $f'$ and $f''$, that is, we have a short exact sequence
$$
\begin{CD} 
0 @>>> M @>>> M'\oplus M'' @>{(\begin{smallmatrix} f'& -f''\end{smallmatrix})}>>
\Delta(i)^{d_i} @>>> 0.
\end{CD}
$$
Similarly, if $i$ is a source, we have short exact sequences
$$
\begin{CD}
0 @>>> \Delta(i)^{d_i} @>{f'}>>  M' @>>> \cok(f') @>>> 0
\end{CD}
$$ 
and 
$$
\begin{CD}
0 @>>> \Delta(i)^{d_i} @>{f''}>>  M'' @>>> \cok(f'') @>>> 0
\end{CD}
$$ 
Let $M$ be the pushout of $f'$ and $f''$, that
is, we have a short exact sequence 
$$
\begin{CD}
0 @>>> \Delta(i)^{d_i}
@>{\bigl(\begin{smallmatrix}f' \\  -f''
\end{smallmatrix}\bigr)}>> M'\oplus M'' @>>> M @>>> 0.
\end{CD}
$$

In both of these cases, we say that $M$ is obtained by gluing $M'$ and
$M''$ at $i$.

\begin{lemma} Let $M'$ and $M''$ be as above.
If $M$ is glued from $M'$ and $M''$ at $i$, then $M$ is $\Delta$-filtered.
\end{lemma}
\begin{proof}
If $i$ is a sink, then we have an exact sequence
$$
\begin{CD}
\ext^2_{\bf D}(M'\oplus M'',-) @>>> \ext^2_{\bf D}(M,-)
@>>> \ext^3_{\bf D}(\Delta(i)^{d_i},-),
\end{CD}
$$
showing that $M$ has projective dimension at most one. Therefore $M$
is $\Delta$-filtered by Proposition \ref{HVTheorem2}.

Now suppose $i$ is a source. Then we have an exact sequence
$$
\begin{CD}
\ext^1_{\bf D}(\Delta(i)^{d_i},-) @>>> \ext^2_{\bf D}(M,-) @>>> \ext^2_{\bf D}(M'\oplus M'',-).
\end{CD}
$$
This shows again that $M$ is $\Delta$-filtered.
\end{proof}

\begin{proposition} \label{allmodulesareglued}
Let $i$ be a sink or a source in $Q$. Then for any $\Delta$-filtered
module $M$ with $(\underline{\dim}_\Delta(M))_i>0$, there exists $\Delta$-filtered modules $M'$ and $M''$ with
$(\underline{\dim}_\Delta(M'))_i=(\underline{\dim}_\Delta(M''))_i=(\underline{\dim}_\Delta(M))_i$ such that $M$
is isomorphic to a module glued from  $M'$ and $M''$ at $i$.
\end{proposition}
\begin{proof}
Let $i$ be a source in $Q$. By Lemma \ref{subsource}, there is a
short exact sequence
$$
\begin{CD}
0 @>>> \Delta(i)^{d_i} @>{f}>> M @>{g}>>  \cok (f) @>>> 0,
\end{CD}
$$
where $\cok (f)$ is
$\Delta$-filtered. Here $\Delta(i)^{d_i}$ is the submodule of $M$
generated by $M_i$. Then $(\underline{\dim}(\cok (f)))_i=0$ and
therefore $\cok (f)= Y'\oplus Y''$, where $\supp(Y')$ is contained
in $\{1,\dots,i-1\}$ and $\supp(Y'')$ is contained in
$\{i+1,\dots,n\}$.
We have short exact sequences
$$
\begin{CD}
0 @>>>
\Delta(i)^{d_i} @>>> M' @>>> Y' @>>> 0
\end{CD}
$$
and
$$
\begin{CD}
0 @>>> \Delta(i)^{d_i} @>>> M'' @>>> Y''
@>>> 0,
\end{CD}
$$
where $M'=g^{-1}(Y')$ and $M''=g^{-1}(Y'')$. By Proposition
\ref{HVTheorem2}, we see that $M'$ and $M''$ are $\Delta$-filtered
with their $\Delta$-supports contained in $\{1,\dots,i\}$ and
$\{i,\dots,n\}$, respectively. Now there is a short exact sequence
$$
\begin{CD}
0 @>>> X @>>> M' \oplus M'' @>{(\begin{smallmatrix} g'& g''\end{smallmatrix})}>>
 M @>>> 0,
\end{CD}
$$
where $g'$ and $g''$ are the inclusion maps. We have
$\underline{\dim}(X)=\underline{\dim}(\Delta(i)^{d_i})$  and $X$ is projective as an $\field Q$-module.
Therefore $X\cong \Delta(i)^{d_i}$. Hence $M$ is glued from $M'$
and $M''$ at $i$.

Now suppose that $i$ is a sink. We let $N'$ and $N''$ be the
submodules of $M$ generated by $V''=\bigoplus^{i-1}_{j=1}M_j$ and
$V'=\bigoplus^n_{j=i+1}M_j$. By the relations in $\mathcal{I}$, we
know that $N''$ is supported on the vertices $\{1,\dots,i\}$ and
$N'$ is supported on the vertices $\{i,\dots,n\}$. Now let $M'=M/N'$
and $M''=M/N''$. We have $(N'\cap N'')_i=0$, since $M$ is
$\Delta$-filtered and so $\im (M_{\alpha_{i-1}})\cap \im
(M_{\alpha_{i+1}})=0 $. Moreover, for all $j\neq i$ we have $N'_j=0$
or $N''_j=0$. Hence $N'\cap N''=0$. So we have an embedding
$$
\begin{CD}
M @>{\bigl(\begin{smallmatrix}f' \\
    f''\end{smallmatrix}\bigr)}>> M'\oplus M'',
\end{CD}
$$
where $f'$ and $f''$ are the quotient maps. Now by comparing
dimension vectors, we see that the cokernel of $M \lra M'\oplus M''$ is
$L(i)^{d_i}\cong \Delta(i)^{d_i}$. Hence $M'\oplus M''$ is
$\Delta$-filtered and $M$ is obtained by gluing $M'$ and $M''$ at $i$.
\end{proof}

From the  proof of Proposition \ref{allmodulesareglued}, for a module $M$  obtained by gluing two
$\Delta$-filtered modules $M'$ and $M''$, we see that if $i$ is a source, then $M'$ and $M''$ can be
viewed as submodules of $M$,
and if $i$ is a sink, then $M'$ and $M''$
can be viewed as quotients of $M$.

\begin{example}
We consider the quiver $Q$ in Example \ref{example1}. The projective module $P(4)$ has two submodules
isomorphic to $P(3)$ and $P(5)$, respectively. We have a short exact sequence
$$
\begin{CD}
0 @>>> P(4)@>>> P(4)/P(3)\oplus P(4)/P(5)@>>> \Delta(4)@>>> 0, 
\end{CD}
$$
that is, the module $P(4)$ can be glued from its quotient modules $ P(4)/P(3)$ and $ P(4)/P(5)$ at vertex $4$.
\end{example}

\begin{lemma}\label{gluingmaps}
Let $h:M\lra N$ be a morphism of two $\Delta$-filtered modules which are
obtained by gluing $M'$, $M''$ and $N'$, $N''$, respectively, at $i$.
Then
\begin{itemize}
\item[(1)] if $i$ is a source, then $f$ restricts to maps $h': M'\lra N'$ and
$h'': M''\lra N''$ of submodules.
\item[(2)] if $i$ is a sink, then $f$
induces maps $h': M'\lra N'$ and $h'': M''\lra N''$ of quotient modules.
\end{itemize}
\end{lemma}
\begin{proof}
Assume that $i$ is a source. We have $h(M')\subseteq N'$ and $h(M'')\subseteq
N''$,  since $h(M_i)\subseteq N_i$. The lemma follows in this case.

Now suppose that $i$ is a sink. We may assume that $M'=M/X$ and
$M''=M/Y$, where $X$ and $Y$ are the submodules of $M$ generated
by $\bigoplus_{j=i+1}^nM_{j}$ and $\bigoplus_{j=1}^{i-1}M_{j}$, respectively. We have $h(M_{i-1})\subseteq N_{i-1}$ and
$h(M_{i+1})\subseteq N_{i+1}$. The lemma follows.
\end{proof}

Let $M$ and $N$ be modules glued from $M'$, $M''$ and $N'$, $N''$, respectively.
If $i$ is a sink, we  construct a morphism $h:M\lra N$ from maps $h': M'\lra N'$ and
$h'': M''\lra N''$ provided there exists a map $w: \Delta(i)^{d_i}\lra
\Delta(i)^{d_i}$, where $d_i=\underline{\dim}_\Delta(M)_i=\underline{\dim}_\Delta(N)_i$, such that the second square of
the diagram of gluing sequences
$$\xymatrix{0\ar[r] & M \ar[r]\ar@{.>}[d] & M'\oplus M''\ar[d]^{h'\oplus h''} \ar[r] &
  \Delta(i)^{d_i} \ar[d]^{w} \ar[r] & 0 \\ 0 \ar[r] & N \ar[r] & N'\oplus N'' \ar[r] &
  \Delta(i)^{d_i} \ar[r] & 0}$$
commutes. Now $h$ is a morphism $M\lra N $, which makes the first square of the above
diagram commute. In this case, we say that $h$ is obtained by gluing $h'$ and
$h''$ at $i$. We similarly define gluing of morphisms at a source $i$.

\begin{lemma} \label{gluingexceptional}
Let $M$ be obtained by gluing $M'$ and $M''$ at $i$. If $M$ is
exceptional, then $M'$ and $M''$ are exceptional.
\end{lemma}
\begin{proof}
Assume that $M$ is exceptional. We first consider the case where $i$
is a sink. We have a short exact sequence 
$$
\begin{CD}
0 @>>> M @>>> M'\oplus M'' @>>> \Delta(i)^c @>>> 0, 
\end{CD}
$$  
where
$c=(\underline{\dim}_\Delta(M'))_i=(\underline{\dim}_\Delta(M''))_i$. By
applying the functor $\hom_{\bf D}(M,-)$ to this sequence and using
Lemma \ref{extinjective}, we see that $\ext^1_{\bf D}(M,M'\oplus
M'')=0$. Then by applying $\hom_{\bf D}(-,M')$, we get a surjection
$\ext^1_{\bf D}(\Delta(i)^c,M')\lra \ext^1_{\bf D}(M'\oplus
M'',M')$. From Lemma \ref{factorsink}, we have a short exact
sequence 
$$
\begin{CD}
0 @>>> X @>>> M'' @>>> \Delta(i)^c
@>>> 0.
\end{CD}
$$ 
Note that $\supp(\mathrm{top}(X))\cap
\supp(M')=\emptyset$, where $\mathrm{top}(X)$ is the top of $X$.
Thus $\hom_{\bf D}(X, M')$\\$=0$. We get an injection $\ext^1_{\bf
D}(\Delta(i)^c,M')\lra \ext^1_{\bf D}(M'',M')$ by  applying
$\hom_{\bf D}(-,M')$. Hence $\ext^1_{\bf D}(M',M')=0$. By symmetry,
$\ext^1_{\bf D}(M'',M'')=0$.


Now assume that $i$ is a source. We have a short exact sequence
$$
\begin{CD}
0 @>>> \Delta(i)^c @>>> M'\oplus M'' @>>> M @>>> 0.
\end{CD}
$$ 
Applying
$\hom_{\bf D}(-,M)$, we get $\ext^1_{\bf D}(M'\oplus M'',M)=0$. Now
by applying $\hom_{\bf D}(M',-)$, we get a surjection $\ext^1_{\bf
D}(M',\Delta(i)^c)\lra \ext^1_{\bf D}(M',M'\oplus M'')$. From Lemma
\ref{subsource}, we have a short exact sequence 
$$
\begin{CD}
0
@>>> \Delta(i)^c @>>> M'' @>>> Y @>>> 0.
\end{CD}
$$ 
Since 
$\supp(M')\cap \supp(Y)=\emptyset$, we have $\hom_{\bf
D}(M', Y)=0$. Using $\hom_{\bf D}(M',-)$, we get an injection
$\ext^1_{\bf D}(M',\Delta(i)^c)\lra \ext^1_{\bf D}(M',M'')$. Hence
$\ext^1_{\bf D}(M',M')=0$. By symmetry, $\ext^1_{\bf D}(M'',M'')=0$.
This finishes the proof.
\end{proof}

\section{Exceptional $\Delta$-filtered modules with a linear $\Delta$-support}
\label{exceptionallinear:section}


In this section, we first recall from \cite{BHRR} the construction
of exceptional $\Delta$-filtered modules with $\Delta$-dimension
vector a non-zero vector in $\mathbb{N}^n$ where $Q$ is linearly
oriented. We modify this construction to the case where the
orientation of $Q$ is arbitrary and the given dimension vector has a
linear support, that is, as a subquiver of $Q$ the support is
linearly oriented. In particular, we say that $M$ is a module with
linear $\Delta$-support if the subquiver $\supp_\Delta(M)$ is
linearly oriented. We will define an order on $\Delta$-filtered
modules.  We also prove some properties of the exceptional modules
with a linear $\Delta$-support.

\subsection{Exceptional $\Delta$-filtered modules for $Q$ of linear orientation}\label{linear}

Following \cite{BHRR}, where $Q$ is of linear orientation with vertex $1$ a sink
and vertex $n$ a source, a module $M$ is isomorphic to a nonzero submodule of $P(1)$
if and only if the socle of $M$ is $L(1)$. Such a module $M$ is indecomposable,
exceptional and $\Delta$-filtered. Moreover, the map sending a submodule $M$ of
$P(1)$ to the set of vertices of its $\Delta$-support affords
a bijection between the set of submodules of
$P(1)$ and the set of the subsets of $\{1, \dots, n\}$. We denote by $\Delta(I)$ the
submodule of $P(1)$ with $I$ as the set of vertices of its $\Delta$-support.
Given $\dv d$ a non-zero vector in $\mathbb{N}^n$ , define $I(\dv d)=\supp(\dv d)_0$, the set
of vertices of $\supp(\dv d)$. We denote
by $\dv d_{I(\dv d)}$ the dimension vector with $$(\dv d_{I(\dv d)})_i=
\left\{\begin{array}{ll} 1& \mbox{ if } i\in I(\dv d), \\ 0 & \mbox{ otherwise}.  \end{array}\right.$$
Now define inductively a module $\Delta(\dv d)$ as follows:
$\Delta(\dv d)=\Delta(I(\dv d))\oplus \Delta(\dv d-\dv d_{I(\dv d)})$. In this way we also obtain a descending sequence of
subsets  $I(\dv d)\supseteq I(\dv d-\dv d_{I(\dv d)})\supseteq \cdots$.

\begin{theorem} \cite{BHRR}\label{bhrrlinear}
Given $\dv d$ a non-zero vector in $\mathbb{N}^n$. The module $\Delta(\dv d)$ constructed above is
the unique (up to isomorphism) exceptional $\Delta$-filtered module  with
$\Delta$-dimension vector $\dv d$.
\end{theorem}

\subsection{Exceptional $\Delta$-filtered modules with a linear $\Delta$-support}\label{sectlinearsupp}
Let $\dv d$ be a non-zero vector in $\mathbb{N}^n$ with a linear support.
Suppose that the set of vertices of $\supp_\Delta(\dv d)$ is  $\{i, i+1, \cdots, j\}$,
where any vertex in $\{i+1, \cdots, j-1\}$ is a non-admissible vertex of $Q$. We
may suppose that $i$ is a source in $\supp(\dv d)$ and $j$ is a sink in $\supp(\dv d)$. We consider
the following subquiver of $Q$,
$$
\xymatrix{l &\ar[l]\dots & i'\ar[l] \ar[r]&\cdots \ar[r]& i
\ar[r] &i+1 \ar[r]&\dots \ar[r] &j \ar[r]&\cdots \ar[r]&j' }, 
$$
where $i'$ is a source vertex, $l$ and $j'$ are sink vertices and
all other vertices are non-admissible vertices of $Q$. We denote by
$Q'$ the subquiver with the set of vertices $Q'_0=\{i', i'+1. \dots,
j, \cdots, j'\}$. Let $\tilde{Q'}$ be the double of $Q'$ and
$\mathcal{I}'=k\tilde{Q}'\cap \mathcal{I}$. Then
$k\tilde{Q'}/\mathcal{I}'$ is quasi-hereditary as well and its Verma
modules are given by the indecomposable projective $kQ'$-modules.
Note that $Q'$ is linearly oriented of type $A$. By Theorem
\ref{bhrrlinear}, there is a unique $\Delta$-filtered
$k\tilde{Q'}/\mathcal{I}'$-module with $\Delta$-dimension vector
$\dv d$. We denote this $ k\tilde{Q'}/\mathcal{I}'$-module again by
$\Delta(\dv d)$. Note that we can consider $\Delta(\dv d)$ as a
${\bf D}$-module which is not necessarily $\Delta$-filtered. Note
also that all indecomposable projective $kQ'$-modules, except at
$i'$ when $i'$ is interior, coincide respectively to the
indecomposable projective $kQ$-modules with the same simple top.
We construct a ${\bf D}$-module $M(\dv d)$ as follows. If $i$ is not
an interior source of $Q$, we let $M(\dv d)=\Delta(\dv d)$. We now
consider the case where $i$ is an interior source  of $Q$. We write

$$ \Delta(\dv d)=\bigoplus_{s\geq 1}\Delta(I_s), $$
where $I(\dv d)=I_1\supseteq I_2=I(\dv d-\dv d_{I(\dv d)})\supseteq \cdots$ and each $\Delta(I_s)$ is an
indecomposable exceptional $\Delta$-filtered $k\tilde{Q}'/\mathcal{I}'$-module.
For each $I_s$,
define $M_{I_s}=\Delta(I_s)$ if $(\underline{\dim}_\Delta(\Delta(I_s)))_i=0$; otherwise
define $M_{I_s}$ as follows,

\begin{itemize}
\item[(1)] $(M_{I_s})_r=(\Delta(I_s))_r$ for $r\in \{i, \cdots, j'\}$ and
$(M_{I_s})_\gamma=(\Delta(I_s))_\gamma$ for $\gamma\in \{\alpha_s,
\beta_s\}_{s=i}^{j'-1}$,
\item[(2)] $(M_{I_s})_r=(\Delta(I_s))_i$ for $r\in \{l,\cdots, i-1\}$ and
$(M_{I_s})_\gamma=1$ for $\gamma\in \{\alpha_s\}_{s=l}^{i-1}$ and
$(M_{I_s})_\gamma=0$ for $\gamma\in \{ \beta_s\}_{s=l}^{i-1}$,
\end{itemize}
and zero elsewhere. By the construction, we see that $M_{I_s}$ is
$\Delta$-filtered and the set of vertices of $\supp_\Delta(M_{I_s})$
is $I_s$. Now let
$$M(\dv d)=\bigoplus _{s\geq 1} M_{I_s}.$$

\begin{example}
We consider the quivers and use the notation in Example \ref{example1}.
Let $\dv d=(0, 1,1,2,0)$. Then
\begin{itemize}
\item[(1)]  $\supp(\dv d)=\xymatrix{2\ar[r]&3\ar[r]&4}$, which is linearly
oriented.

\item[(2)] Denote $I(\dv d)$ by $I_1$ and so $I_1=\{2, 3, 4\}$ and $\dv d_{I_1}=(0, 1,1,1, 0)$.
Thus $\dv d-\dv d_{I_1}=(0, 0,0, 1,0)$ and $I(\dv d- \dv d_{I_1})=\{4\}$, we denote it by $I_2$. Now
$\Delta(\dv d)=\Delta(I_1)\oplus \Delta(I_2)$ with $\Delta(I_2)=L(4)$ and $\Delta(I_1)$ as follows.
$$
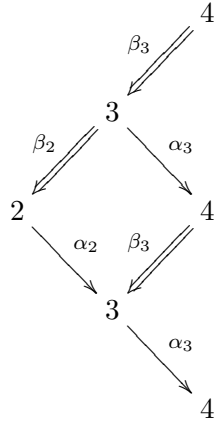
\xymatrix{ &&& 4\ar@{=>}[dl]_{\beta_3}&&&  \\ && 3\ar[dr]^{\alpha_3}
  \ar@{=>}[dl]_{\beta_2} && \\ &
  2\ar[dr]^{\alpha_2}  && 4 \ar@{=>}[dl]_{\beta_3} & & \\ &&3\ar[dr]^{\alpha_3} \\ &&& 4 }$$
\end{itemize}
Note that $\Delta(I_1)$ is a $\Delta$-filtered
$k\tilde{Q'}/\mathcal{I}'$-module, but not a $\Delta$-filtered ${\bf
D}$-module, where $Q'=\supp(\dv d)$ and
$\mathcal{I'}=k\tilde{Q'}\cap \mathcal{I}$.

\item[(3)]By our construction
above we have $M_{I_1}$ as follows.
$$
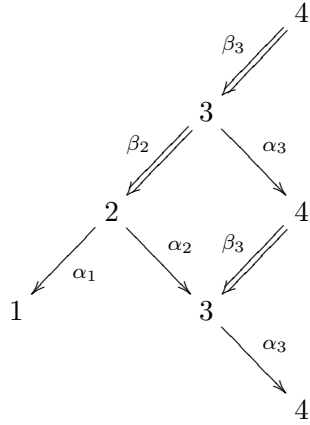
\xymatrix{ &&& 4\ar@{=>}[dl]_{\beta_3}&&&  \\ && 3\ar[dr]^{\alpha_3}
  \ar@{=>}[dl]_{\beta_2} && \\ &
  2\ar[dr]^{\alpha_2}\ar[dl]^{\alpha_1}  && 4 \ar@{=>}[dl]_{\beta_3} & & \\ 1&&3\ar[dr]^{\alpha_3} \\ &&& 4 }$$
Since $(\underline{\dim}_\Delta(\Delta(I_2)))_2=0$, we let $M_{I_2}=\Delta(I_2)$.
Now $M(\dv d)=M_{I_1}\oplus L(4)$.
\end{example}

\begin{proposition} \label{linearsupport}
Let $\dv d$ be a non-zero vector in $\mathbb{N}^n$ with a linear support as above. Then the
module $M(\dv d)$ constructed above is the unique (up to isomorphism)
exceptional $\Delta$-filtered module with $\Delta$-dimension vector
$\dv d$.
\end{proposition}
\begin{proof} 
In view of Proposition \ref{irreducibility},  we need only to show
that the module $M(\dv d)$ is exceptional.  Suppose that $L$ is a
self-extension of $M(\dv d)$. Then $L$ is $\Delta$-filtered. If $i$
is not an interior source of $Q$, then this proposition follows from
Theorem \ref{bhrrlinear}. We consider the case where $i=i'$ is
interior. Since $(\underline{\dim}_\Delta M(\dv
d))_{i}=\underline{\dim}( M(\dv d))_{i}=d_i$, we have $M(\dv
d)_{\alpha_{r}}$ is injective for $l\leq r\leq i-1$, and so is
 $L_{\alpha_r}$ for $l\leq r\leq i-1$.
We have a short exact sequence
$$
\begin{CD}
0 @>>> M(\dv d) @>{\lambda}>>
L @>{\mu}>> M(\dv d) @>>> 0,
\end{CD}
$$
which induces another short exact sequence:
$$
\begin{CD}
0 @>>> \Delta(\dv d) @>{\overline{\lambda}}>>
L/\Delta(i-1)^{2d_{i}} @>{\overline{\mu}}>> \Delta(\dv d) @>>> 0.
\end{CD}
$$
Since $\Delta(\dv d)$ is exceptional, there is a morphism
$\overline{\eta}: L/\Delta(i-1)^{2d_{i}}\lra \Delta(\dv d)$ such
that $\overline{\eta}\overline{\lambda}=Id_{\Delta(\dv d)}$. Now let
$\eta: L\lra M(\dv d)$ be defined by $\eta_r=(\overline{\eta}_r)$
for $i\leq r\leq j'$ and $\eta_{r}=M_{\alpha_r}{\eta}_{r+1}
L_{\alpha_{r}}^{-1}: L_{r}\lra M(\dv d)_{r}$ for $l\leq r\leq i-1$.
We can check that $\eta$ is a morphism from $L$ to $M(\dv d)$ and
$\eta\lambda=Id_{M(\dv d)}$. Thus $L$ is a trivial self-extension of
$M(\dv d)$, and so $M(\dv d)$ is exceptional. This finishes the
proof.
\end{proof}

As a corollary of Proposition \ref{linearsupport}, we have the following
result which is our version of Proposition 1 in \cite {BHRR} for an
exceptional $\Delta$-filtered module with linear $\Delta$-support.

\begin{corollary}\label{extensionlinear} We use the same notation as above.
Let $I$ be a subset of $Q_0$ which is contained in a subquiver  with
linear orientation. Then for any $s\in I$, we have
$\ext^1_{\bf D}(M_I, \Delta(s))=0=\ext^1_{\bf D}(\Delta(s), M_I)$.
\end{corollary}


By the construction of exceptional $\Delta$-filtered modules with a linear $\Delta$-support and
Proposition \ref{linearsupport},
we have a bijection between the set of the indecomposable exceptional $\Delta$-filtered modules
with their $\Delta$-supports contained in $Q'$
and
\[
\mbox{the set of the submodules of}\left\{ \begin{array}{ll} P(j')/P(j'+1) & \mbox{if } j' \mbox{ is interior}, \\
P(j')& \mbox{if } j' \mbox{ is not interior}. \end{array} \right.
\]
 Moreover,
$M_J/S_J = \Delta(J)$, where $J\subseteq \{i', \dots, j'\}$ and
$S_J$ is the submodule of $M_J$, generated by $(M_J)_{i'-1}$. We
also note that, as in \cite{BHRR}, $J\mapsto M_J$ induces a
bijection between non-empty subsets $J\subseteq \{i',\dots,j'\}$ and
isomorphism classes of indecomposable exceptional $\Delta$-filtered
modules with their $\Delta$-supports contained in $Q'$. Note that
this bijection $J\mapsto M_J$ holds too, if $Q'$ has an opposite
orientation.

\subsection{An order on $\Delta$-filtered modules}

In the next section, we will  consider  gluing exceptional modules
with linear $\Delta$-support to form exceptional $\Delta$-filtered
modules with arbitrary $\Delta$-support. The construction will
depend on an order on $\Delta$-filtered modules. In this subsection,
we define this order and prove some  preliminary results.

Let $M$ and $N$ be two $\Delta$-filtered modules with
$(\underline{\dim}_\Delta(M))_i=1=(\underline{\dim}_\Delta(N))_i$ for an admissible vertex $i$. If $i$
is a source, there is an inclusion $\Delta(i) \lra M$ by Lemma
\ref{subsource}, and we define $M\geq_iN$ if the map $\hom(M,N) \lra
\hom(\Delta(i),N) $ is surjective. That is, $M\geq_i N$ if for each
map $f: \Delta(i)\lra N$ there is a map $h: M\lra N$ such that the
diagram
$$\xymatrix{0 \ar[r] & \Delta(i)
  \ar[r] \ar[dr]^f & M \ar[d]^h \\  & & N}$$
commutes. The order does not depend on any particular choice of
inclusion $\Delta(i)\lra M$.

Similarly, if $i$ is a sink, there is a quotient map
$N\lra \Delta(i)$ by Lemma \ref{factorsink}, and we define
$M\geq_i N$, if $\hom(M,N)\lra \hom(M,\Delta(i))$ is surjective.
By $M>_i N$ we mean that $M\geq_i N$ and $N\not\geq_i M$.

These orders are transitive and reflexive on the isomorphism classes
of $\Delta$-filtered modules $M$ with $(\underline{\dim}_\Delta(M))_i=1$. Here,
transitivity is trivial and reflexivity follows from Lemma
\ref{subsource} and Lemma \ref{factorsink}.
We compute these orders for the indecomposable exceptional $\Delta$-filtered
modules $M_J$ with $J$ contained in the linearly oriented subquiver $Q'$ as in
Section \ref{sectlinearsupp}. To simplify the notation we may assume that $i=i'$ and
$j=j'$ are a source vertex and a sink vertex of $Q$, respectively.
Similar to the proof of Proposition \ref{linearsupport}, we can prove the
following using Lemma 4 in \cite{BHRR}.

\begin{lemma} \label{lemma4}
Let $J=\{j_1\succ j_2\succ\dots\succ j_s\}$ and $J'=\{j'_1\succ j'_2\succ\dots\succ j'_t\}$ be
two subsets of $\{i,\dots,j\}$. Then the following conditions are
equivalent:
\begin{itemize}
\item[(i)] There is a monomorphism $M_J\lra M_{J'}$.
\item[(ii)] We have $s\leq t$ and $j_r\preceq j'_r$ for $r=1,\dots,s$.
\end{itemize}
\end{lemma}

\begin{lemma} \label{linearsourceorder}
Let $J=\{j_1\succ j_2>\dots \succ j_s\}$ and $J'=\{j'_1\succ j'_2 \succ\dots \succ j'_t\}$ be
two subsets of $\{i,\dots,j\}$, where $j_1=i=j'_1$ is a source in
$Q$. Then $M_{J}\geq_{i} M_{J'}$ if and only if $s\leq t$ and
$j_r\preceq j'_r$ for $r=1,\dots,s$.
\end{lemma}
\begin{proof}
We have $M_J\geq_i M_{J'}$ if and only if there exists a
monomorphism $M_J\lra M_{J'}$, since for any morphism $f$ from $M_J$
to $M_{J'}$,  $\soc(M_J)\subseteq \Delta(i)$ and $\ker(f)\cap
\soc(M_J)\neq 0$ if $\ker(f)\neq 0$. Now the lemma follows from our
construction of $M_J$ and $M_{J'}$ and Lemma \ref{lemma4}.
\end{proof}

\begin{lemma} \label{linearsinkorder}
Let $J=\{j_1\prec j_2\prec \dots\prec j_s\}$ and $J'=\{j'_1 \prec j'_2\prec \dots\prec j'_t\}$ be
two subsets of $\{i,\dots,j\}$, where $j_1=j=j'_1$ is a sink in $Q$.
Then $M_{J}\geq_j M_{J'}$ if and only if  $s\geq t$ and $j_{r}\preceq
j'_{r}$ for $r=1,\dots,t$.
\end{lemma}
\begin{proof}
From Lemma \ref{factorsink}, we have a surjection $M_J\lra
\Delta(j)$. Assume that this surjection factors through a map $f:M_J\lra
M_{J'}$. By Proposition \ref{HVTheorem1},
we have $\im(f)$ is a $\Delta$-filtered submodule of $M_{J'}$.
Thus $\ker(f)$ is $\Delta$-filtered, following again from Proposition \ref{HVTheorem1}.
Since $\ker(f)$ is
$\Delta$-filtered, by Lemma \ref{linearsourceorder},  $\ker(f)=M_{\{j_{s'+1},\dots,j_s\}}$
for some $s'\in \{1,\dots,s\}$. Hence
$\im(f)=M_{\{j_1,\dots,j_{s'}\}}$. Using  Lemma \ref{lemma4}, we have
$s'\leq t$ and $j_{s'-i}\preceq j'_{t-i}$ for $i=0,\dots,s'-1$. But
since $f$ maps the top $\Delta(j)$ in $M_J$ to the top $\Delta(j)$ in
$M_{J'}$, we see that $s'=t$. Therefore $s\geq t$ and $j_i\preceq j'_i$
for $i=1,\dots,t$.

The converse follows from the construction of $M_J$ and $M_{J'}$.
\end{proof}

As a corollary of Proposition \ref{linearsupport} and Lemmas \ref{linearsourceorder}
 and \ref{linearsinkorder}, we have the following:

\begin{corollary}\label{lineartotalorder}
Let $\dv d$ be a non-zero vector in $\mathbb{N}^n$ with $\supp(\dv d)_0\subseteq \{i, \cdots, j\}$.
Then the indecomposable direct summands $X$ of the exceptional $\Delta$-filtered
module $M(\dv d)$ with $(\underline{\dim}_\Delta(X))_i=1$ are totally ordered using
$\geq_i$. Similarly, we have a total order using  $\geq_j$.
\end{corollary}

In the following we give some examples on the order defined above.
\begin{example} We consider the quiver and use the notation in Example
\ref{example1}.
Let $M$ and $N$, respectively,  be the modules as follows.

$$
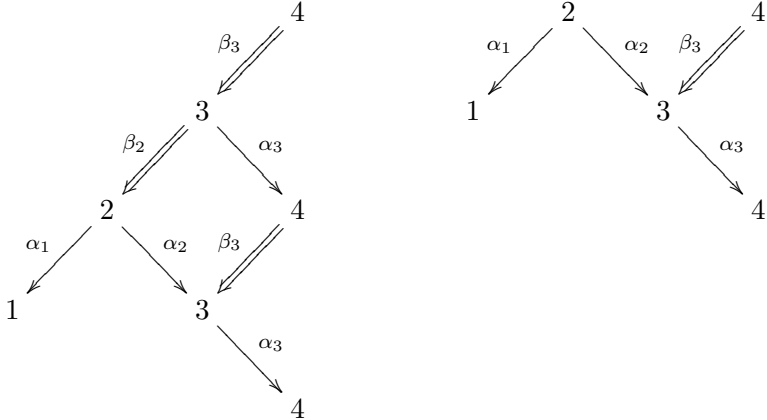
\xymatrix{ &&& 4\ar@{=>}[dl]_{\beta_3} &&&
2\ar[dl]_{\alpha_1}\ar[dr]^{\alpha_2}
&& 4\ar@{=>}[dl]_{\beta_3}
\\ && 3\ar[dr]^{\alpha_3}
  \ar@{=>}[dl]_{\beta_2} && & 1&&3\ar[dr]^{\alpha_3} \\ &
  2\ar[dr]^{\alpha_2} \ar[dl]_{\alpha_1} && 4 \ar@{=>}[dl]_{\beta_3} & & & &&4 &&\\ 1&&3\ar[dr]^{\alpha_3} \\ &&& 4 }$$
It is clear that we have an embedding $N$ into $M$, which sends the submodule $\Delta(2)$ of $N$ to
the submodule $\Delta(2)$ of $M$. On the other hand, we have a morphism from $M$ to $N$, which sends the top $\Delta(4)$ of $M$ to
the top $\Delta(4)$ of $N$. Therefore $M\geq_4 N$ and $N\geq_2 M$.
\end{example}

\section{Exceptional $\Delta$-filtered modules}
\label{exceptional:section}

This section is devoted to proving the following main result.

\begin{theorem}\label{maintheorem1}
Given a non-zero vector $\dv d\in \mathbb{N}^n$, there exists a unique (up to
isomorphism) exceptional $\Delta$-filtered  ${\bf D}$-module $M(\dv d)$ with $\Delta$-dimension
vector $\dv d$.
\end{theorem}

Let $\dv d$ be a non-zero vector in $\mathbb{N}^n$. We construct an
exceptional $\Delta$-filtered representation $M(\dv d)$, with
$\Delta$-dimension vector $\dv d$. Following Lemma
\ref{gluingexceptional}, we see that any indecomposable exceptional
module is obtained by gluing exceptional modules with linear
$\Delta$-support. In the following we show how to glue the
exceptional $\Delta$-filtered modules with linear $\Delta$-support
to obtain exceptional $\Delta$-filtered modules with arbitrary
$\Delta$-support.


Let $i_1<i_2\cdots <i_t$ be a complete list of interior admissible
vertices in $Q$. Let $i_0=1$ and
$i_{t+1}=n$ be the end vertices of $Q$. Let $\dv d^s$, for $s=1,\dots,t+1$, be the
vector given by $(\dv d^s)_j=d_j$ if
$j\in\{i_{s-1},\dots,i_s\}$ and zero elsewhere. Note that each support $\supp(\dv d^s)$
is a linearly oriented subquiver of $Q$.

Let $M(\dv d^s)=\bigoplus_lM_{J^s_l}$ be the exceptional
$\Delta$-filtered module from Proposition \ref{linearsupport}. We
fix an ordering on the indecomposable direct summands of $M(\dv
d^s)$ as follows. For $s=1$, we assume that
$(\underline{\dim}_\Delta(M_{J^1_l}))_{i_1}=1$ for $l=1,\dots,d_{i_1}$
and that $M_{J^1_l}\geq_{i_1} M_{J^1_{l+1}}$ for
$l=1,\dots,d_{i_1}-1$. For each $s>1$, we assume that
$(\underline{\dim}_\Delta(M_{J^{s}_l}))_{i_{s-1}}=1$ for
$l=1,\dots,d_{i_{s-1}}$ and that
$M_{J^{s}_l}\geq_{i_{s-1}}M_{J^{s}_{l+1}}$ for
$l=1,\dots,d_{i_{s-1}}$. This is possible by Lemmas
\ref{linearsourceorder} and \ref{linearsinkorder}, and the fact that
for each $s$ the subsets $J^{s}_l$ are totally ordered by inclusion.
For $l>d_{i_s}$, $s\geq 1$, we fix an arbitrary order.

Let $\dv c^s$ be the $\Delta$-dimension vector given by $\dv
(c^s)_j=d_j$ for $j\in \{1,\dots,i_s\}$ and zero elsewhere. Here $\dv
c^{t+1}=\dv d$ and $\dv c^1=\dv d^1$. We will inductively construct
an exceptional $\Delta$-filtered module $M(\dv c^s)$ with
$\Delta$-dimension vector $\dv c^s$ for all $s=1,\dots,t+1$. If
$t=0$, then $Q$ is linearly oriented, and we let $M(\dv
d)=\Delta(\dv d)$ as in \cite{BHRR}. Now suppose that $t>0$. For
$s=1$, we let $M(\dv c^1)=M(\dv d^1)$. Suppose that we have $M(\dv
c^s)$. We construct $M(\dv c^{s+1})$ by gluing $M(\dv c^s)$ and
$M(\dv d^{s+1})$ at vertex $i_s$ as follows.

First, we decompose $M(\dv c^s)$ into indecomposable direct summands
$$M(\dv c^s)=\bigoplus_l M_{K^{s}_l},$$ where  $K^{
s}_l\subseteq\{1,\dots,i_s\}$ and
$(\underline{\dim}_\Delta(M_{K^{s}_l}) )_j=1$ if $j\in K^{ s}_l$ and
zero elsewhere. We will show that $M_{K^{ s}_l}$ is the unique, up
to isomorphism, indecomposable exceptional $\Delta$-filtered module
with  ${K^{s}_l}$ as the set of vertices of its $\Delta$-support.
Moreover, we reorder the indecomposable direct summands of $M(\dv
c^s)$ such that
\begin{itemize}
\item[(1)] $(\underline{\dim}_\Delta( M_{K^s_l} ) )_{i_s}=1$ for $l=1,\dots,d_{i_s}$;
\item[(2)] $M_{K^s_l} \geq_{i_s} M_{K^s_{l+1}}$ for $l=1,\dots,d_{i_s}-1$.
\end{itemize}
Unlike the case of linear support, not all subsets of $\{1,\dots,n\}$
will occur as the support of an indecomposable exceptional
$\Delta$-filtered module.


If $d_{i_s}=0$, by induction we have $M(\dv d)=M(\dv d')\oplus M(\dv d'')$, where 
$({\dv d}')_i=d_i$ and $({\dv d}'')_j=d_j$ for
$1\leq i< i_s$, $i_s<j\leq n$  and zero elsewhere. Now suppose that $d_{i_s}\not= 0$ and define
$M(\dv c^{s+1})$ as
$$M(\dv c^{s+1})=\bigoplus_l M_{K^{ {s+1}}_l},$$ where $K^{{s+1}}_l= K_l^s\cup J_{{d_{i_s}+1-l}}^{s+1}$ and
$M_{K^{{s+1}}_l}$ is obtained by gluing $M_{K^s_l}$ and
$M_{J_{d_{i_s+1-l}}^{s+1}}$ at $i_s$ for $l=1, \dots, d_{i_s}$, and
$\bigoplus_{l>d_{i_s}} M_{K^{{s+1}}_l}$ is the direct sum of all the
terms $M_{K^s_l}$ and $M_{J^{s+1}_l}$ with $l>d_{i_s}$. For
$l>d_{i_s}$, we have that $K^{s+1}_l$ is either $K^s_{l'}$ or
$J^{s+1}_{l''}$ for some $l', \; l''>d_{i_s}$.
By the construction and the properties of gluing, we get the
following proposition.

\begin{proposition}\label{basicproperties}
The representation $M(\dv c^{s})$ constructed above is
$\Delta$-filtered and each of its direct summand $M_{K^{ s}_l}$ is
indecomposable.
\end{proposition}

\begin{example}\label{arbitraygluing}
We consider the quiver and use the notation  in Example
\ref{example1} where $i_1=2$ and $i_2=4$. Let $\dv d=(1, 2, 1, 3,
2)$.
 We decompose $\dv d$ into its  subvectors,
 $\dv d^1=(1, 2, 0, 0, 0)$, $\dv d^2=(0, 2, 1, 3, 0)$ and $\dv d^3=(0,  0, 0,  3, 2)$.
 Each of these subvectors has a linear support.

(1) Then we have $M(\dv d^1)=M_{J^1_1}\oplus M_{J^1_2}$, where $J^1_1=\{ 2\}\subset J^1_2=\{1, 2\}$.
Thus $M_{J^1_1}>_2M_{J^1_2}$ and they are as follows, respectively.
$$
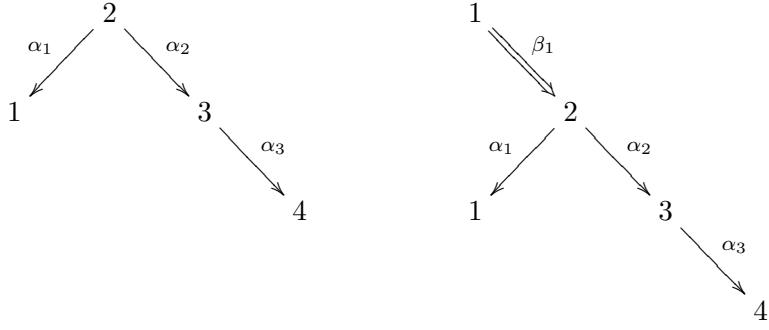
\xymatrix{
&2 \ar[dl]_{\alpha_1}\ar[dr]^{\alpha_2} &&&& 1 \ar@{=>}[dr]^{\beta_1}&&&&\\
1&& 3\ar[dr]^{\alpha_3} &&&& 2 \ar[dl]_{\alpha_1}\ar[dr]^{\alpha_2}\\
&&&4 && 1&&3\ar[dr]^{\alpha_3} \\&&&&&&&&4
}$$

The module $M(\dv d^2)=M_{J^2_1}\oplus M_{J^2_2}\oplus M_{J^2_3}$, where $J^2_1=\{2, 4\}$, $J^2_2=\{2, 3, 4 \}$ and $J^2_3=\{4\}$.
Thus $M_{J^2_1}>_2M_{J^2_2}$ and $M_{J^2_1}$, $M_{J^2_2}$ and $ M_{J^2_3}$ are
as follows, respectively.
$$\xymatrix{
& 2\ar[dl]_{\alpha_1}\ar[dr]^{\alpha_2} &&4\ar@{=>}[dl]_{\beta_3}&&&&& 4\ar@{=>}[dl]_{\beta_3} &&4\\
1&& 3\ar[dr]^{\alpha_3} &&&&& 3\ar[dr]^{\alpha_3} \ar@{=>}[dl]_{\beta_2}&\\
&&&4 &&& 2\ar[dr]^{\alpha_2} \ar[dl]_{\alpha_1} && 4 \ar@{=>}[dl]_{\beta_3}\\
&&&&& 1&&3 \ar[dr]^{\alpha_3} \\&&&&&&&& 4
}$$

The module $M(\dv d^3)=M_{J^3_1}\oplus M_{J^3_2}\oplus M_{J^3_3}$, where $J^3_1=J^3_2=\{4, 5\}$ and $J^3_3=\{4\}$. Thus
$M_{J^3_1}= M_{J^3_2}>_4 M_{J^3_3}$ and they  are as follows, respectively.
$$
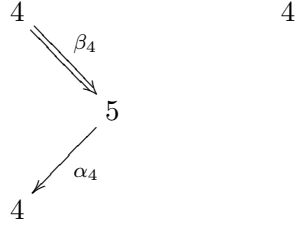
\xymatrix{ 4 \ar@{=>}[dr]^{\beta_4}&&& 4\\
&5 \ar[dl]^{\alpha_4} &&&
\\  4&&&& }$$

(2) Glue $M(\dv d^1)$ and $M(\dv d^2)$ at vertex $2$. That is,  glue $M_{J^1_1}$ and $M_{J^2_2}$, and
glue $M_{J^1_2}$ and $M_{J^2_1}$ at vertex $2$. Thus we obtain $M(\dv c^2)=\bigoplus_{i=1}^3M_{K^2_1}$,
where $K^2_1=J^1_1\cup J^2_2=J^2_2$, $ K^2_2=J^1_2\cup J^2_1=\{1, 2, 4 \}$ and $K^2_3=J^2_3$.
Moreover, $M_{K^2_1}>_4M_{K^2_2}>_4 M_{K^2_3}$.

(3) Finally we have $M(\dv d)=\bigoplus_{i=1}^3 M_{K^3_i}$, where
$K^3_1= K^2_1\cup J^3_3=\{2, 3, 4\}$ and $K^3_2=K^{2}_2\cup J^3_2=\{1, 2, 4, 5\} $ and
$K^3_3=K^2_3\cup J^3_1=J^3_1$, and
$ M_{K^3_1}=M_{J^2_2}$, $M_{K^3_3}=M_{J^3_1}$ and $M_{K^3_2}$ is as follows.

$$
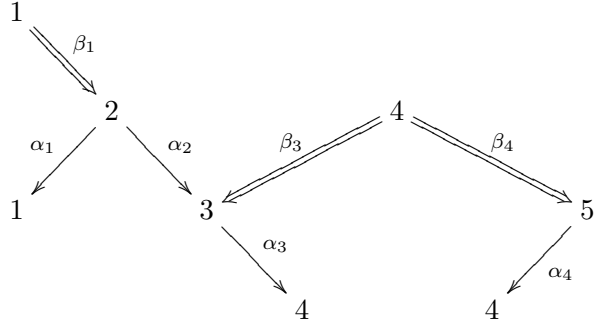
\xymatrix{ 1 \ar@{=>}[dr]^{\beta_1}\\ &
2\ar[dl]_{\alpha_1}\ar[dr]^{\alpha_2}&&& 4\ar@{=>}[dll]_{\beta_3}\ar@{=>}[drr]^{\beta_4}
\\ 1&&3\ar[dr]^{\alpha_3}&&&& 5\ar[dl]^{\alpha_4}\\&&&4 &&4
}$$

(4) These three modules $ M_{K^3_1}$, $M_{K^3_2}$ and $M_{K^3_3}$ are indecomposable.
\end{example}



We need some lemmas for the inductive step. Let us fix some
notation.

Let $M$ and $N$ be $\Delta$-filtered modules obtained by
successively gluing $M^{a}, \dots, M^{u}$ and $N^{b}, \dots, N^{u}$,
respectively, where $1\leq a \leq u \leq t+1$, $1\leq b \leq u $,
and for each $i$, there exist integers $m_i$ and $n_i$ such that
$$
M^i = M_{J^i_{m_i}} \ \mbox{ and } \ N^i = M_{J_{n_i}^i}.
$$
For any $l$, we denote by $M^{\leq l}$ the module obtained by
successively gluing $M^{j}$ for $j=a,\dots,l$, and by $M^{\geq l}$
the module obtained by gluing $M^j$ for $j=l,\dots,u$. We obtain $M$
by gluing $M^{\leq l}$ and $M^{\geq l}$ at $i_l$. Similarly, for
$N^{\leq l}$ and $N^{\geq l}$.

\begin{lemma} \label{generalorder}
Assume
that $J^j_{m_j}\subseteq J^j_{n_j}$ or $J^j_{n_j}\subseteq
J^j_{m_j}$ for $\mathrm{max}\{a, b\}\leq j\leq u$. Then
$M\geq_{i_u}N$ if and only if, either
\begin{itemize}
\item[(1)] $M\cong N$, or
\item[(2)] there exists an $l$ such that $M^l >_{i_l}
  N^l$ and $M^j \cong N^j$ for $j=l+1,\dots,u$.
\end{itemize}
\end{lemma}
\begin{proof}
If $M\cong N$, then clearly $M\geq_{i_u} N$, so we may assume that
$M\not\cong N$. Then there exists an $l$ such that $M^l \not\cong
N^l$ and $M^s\cong N^s$ for $s=l+1,\dots,u$. We need only to prove
the lemma by showing that $M\geq_{i_u}N$ if and only if $M^l>_{i_l}
N^l$.


First assume that $M\geq_{i_u}N$. Let $f: M\lra N$ be a morphism with the
property that a surjective morphism $M\lra \Delta(i_u)$ factors through $f$, if $i_u$
is a sink, and an injective morphism $ \Delta(i_u)\lra N$ factors through $f$, if
$i_u$ is a source. Such a morphism $f$ exists by the definition of
$\geq_{i_u}$.

By Lemma \ref{gluingmaps}, we have a map $g=f|_{M^u}: M^u \lra N^u$,
the restriction of $f$ to $M^u$,  which shows that $M^u\geq_{i_u}
N^u$. If $M^u \not\cong N^u$, then $M^u>_{i_u} N^u$ and we are done.
We will show that if $M^u \cong N^u$, then $g$ is an isomorphism. If
$i_u$ is a source, then $g$ is injective, and therefore an
isomorphism.  If $i_u$ is a sink, then $\im(g)\geq_{i_u}N^u$. So by
Lemma \ref{linearsinkorder}, we see that the $\Delta$-length of
$\im(g)$ is greater than or equal to the $\Delta$-length of $N^u
\cong M^u$. Hence $g$ is surjective and therefore an isomorphism.
Note any $\Delta$-filtered module $L$ with
$(\underline{\dim}_\Delta(L))_i=1$ and
$(\underline{\dim}_\Delta(L))_j=0$ for $j\in \{1, 2, \cdots, i-1\}$,
where $i$ is admissible,  can always be viewed as a gluing of $L$
and $\Delta(i)$ at vertex $i$. And  so if we have $M^{\geq u-1}\cong
N^{\geq u-1}$, we can assume either both $N^{u-1}$ and $ M^{u-1}$
are  zero or both are non-zero. Since $g$ is an isomorphism, we have
$M^{\leq u-1}\geq_{i_{u-1}} N^{\leq u-1}$. Using induction we get
that $M^l >_{i_l} N^l$ and we are done.

For the converse, assume that  $M^l >_{i_l} N^l$ and $M^j\cong N^j$
for $j=l+1,\dots,u$. If $l=u$, let $g:M^u\lra N^u$ be a morphism
with the property that a surjective morphism $M^u\lra \Delta(i_u)$
factors through $g$, if $i_u$ is a sink, and an injective morphism
$\Delta(i_u)\lra N^u$ factors through $g$, if $i_u$ is a source.
Since $M^u\not\cong N^u$ we see that $\Delta(i_{u-1})$ is in the
kernel of $g$, if $i_u$ is a sink, and that the top isomorphic to
$\Delta(i_{u-1})$ is not in the image of $g$, if $i_u$ is a source.
So in both cases, we get a morphism $f: M\lra N$ by gluing $g$ and
the zero map $0: M^{\leq u-1}\lra N^{\leq u-1}$ at $i_{u-1}$. This
shows that $M\geq_{i_u} N$.

If $l<u$, then $M^{\leq u-1}\geq_{i_{u-1}} N^{\leq u-1}$ by
induction and there exists a morphism  $g: M^{\leq u-1}\lra N^{\leq
u-1}$ such that a surjective morphism $M^{\leq u-1}\lra
\Delta(i_{u-1})$ factors through $g$, if $i_u$ is a source, and an
injective morphism $ \Delta(i_{u-1})\lra N^{\leq u-1}$ factors
through $g$, if $i_u$ is a sink. In both cases, we may glue $g$ with
an isomorphism $M^u \lra N^u$ at $i_{u-1}$, and get a map $f: M\lra
N$, which has the property that a surjective morphism $M\lra
\Delta(i_u)$ factors through $f$, if $i_u$ is a sink, and an
injective morphism $\Delta(i_u)\lra N$ factors through $f$, if $i_u$
is a source.  This shows that $M\geq_{i_u}N$.
\end{proof}


\begin{lemma} \label{constructionexceptional}
The module $M\oplus N$ is exceptional if and only if $M^{\leq u-1}\oplus
N^{\leq u-1}$ is exceptional and either
\begin{itemize}
\item[(1)] $M^{\leq u-1}\geq_{i_{u-1}} N^{\leq u-1}$ and $N^u \geq_{i_{u-1}} M^u$, or
\item[(2)] $N^{\leq u-1}\geq_{i_{u-1}} M^{\leq u-1}$ and $M^u \geq_{i_{u-1}} N^u$.
\end{itemize}
\end{lemma}

We need some preparation for the proof of this lemma. To simplify
notation, we let $M'=M^{\leq u-1}$, $M''=M^u$, $N'=N^{\leq u-1}$ and
$N''=N^u$.

\begin{lemma} \label{prooflemma1}
We have $\ext_{\bf D}^1(M', \Delta(i_{u-1}))=0=\ext_{\bf D}^1(\Delta(i_{u-1}), M')$, and likewise for $N'$.
\end{lemma}

\begin{proof}
We first consider the case where $i_{u-1}$ is a source.
Since $\Delta(i_{u-1})$ is a projective ${\bf D}$-module, and so $ \ext_{\bf D}^1(\Delta(i_{u-1}), M')=0$.
Thus we need only to prove
$\ext_{\bf D}^1(M', \Delta(i_{u-1}))=0$.
By the construction of $M'$, we see that $M'$ has a submodule $Y$ with

\[ (\underline{\dim}_\Delta(Y))_i=\left\{\begin{array}{ll}
(\underline{\dim}_\Delta(M'))_i & \mbox{if } 1\leq i\leq  i_{u-2}-1, \\
0 & \mbox{otherwise}. \end{array} \right.\]
We have an exact sequence
$\xymatrix{0\ar[r]& Y\ar[r]& M'\ar[r]& X\ar[r]& 0}$,  where
\[ (\underline{\dim}_\Delta(X))_i=\left\{\begin{array}{ll}
(\underline{\dim}_\Delta(M'))_i & \mbox{if } i_{u-2}\leq i\leq  i_{u-1} , \\
0 & \mbox{otherwise}. \end{array} \right.\]
That is, $\supp_\Delta(X)$ has a linear orientation.
By applying $\hom_{\bf D}(-, \Delta(i_{u-1}))$, we have an exact sequence
$$
\begin{CD}
\ext^1_{\bf D}(X, \Delta(i_{u-1})) @>>> \ext^1_{\bf D}(M', \Delta(i_{u-1}))@>>>
\ext^1_{\bf D}(Y, \Delta(i_{u-1})).
\end{CD}
$$
By Corollary \ref{extensionlinear}, we have $\ext^1_{\bf D}(X, \Delta(i_{u-1}))=0$. Note that
the support of the top of the first syzygy $\Omega(Y)$ of $Y$ is contained in
$\{0, ..., i_{u-2}-1\}$, and so $\hom_{\bf D}(\Omega(Y), \Delta(i_{u-1}))=0$. Thus
$\ext^1_{\bf D}(Y, \Delta(i_{u-1}))=0$, and so $\ext^1_{\bf D}(M', \Delta(i_{u-1}))=0 $.

Now suppose that $i_{u-1}$ is a sink. By Lemma \ref{extinjective},
we have $\ext^1_{\bf D}(M',\Delta(i_{u-1}) )=0 $. Thus we need only
to prove that  $\ext_{\bf D}^1(\Delta(i_{u-1}), M')=0$. By the
construction of $M'$, there exists a submodule $Y$ of $M'$ such that

\[ (\underline{\dim}_\Delta(Y))_i=\left\{\begin{array}{ll}
(\underline{\dim}_\Delta(M'))_i & \mbox{if }i_{u-2}\leq i\leq i_{u-1}, \\
0 & \mbox{otherwise}. \end{array} \right.\]
Note that $Y$ has linear $\Delta$-support. We have
$ \xymatrix{0\ar[r]& Y\ar[r]& M'\ar[r]&X\ar[r]& 0 }$. By applying
$\hom_{\bf D}(\Delta(i_{u-1}), -)$, we get an exact sequence
$$
\begin{CD}
\ext^1_{\bf D}(\Delta(i_{u-1}), Y) @>>> \ext^1_{\bf D}(\Delta(i_{u-1}), M') @>>>
\ext^1_{\bf D}(\Delta(i_{u-1}), X).
\end{CD}
$$
By Corollary \ref{extensionlinear}, we have $\ext^1_{\bf D}(\Delta(i_{u-1}), Y)=0$. Note that
$\Omega(\Delta(i_{u-1}))=P(i_{u-1}-1)\oplus P(i_{u-1}+1)$ and so
$\hom_{\bf D}(\Omega(\Delta(i_{u-1})), X)=0$, since $\supp(X)$ is contained in
$\{1, ..., i_{u-2}-1\}$. Therefore $ \ext^1_{\bf D}(\Delta(i_{u-1}), X)=0$ and
so $\ext^1_{\bf D}(\Delta(i_{u-1}), M')=0$. This finishes the proof.
\end{proof}

\begin{lemma}\label{prooflemma2}
We have $ \ext^1_{\bf D}(M', N'')=0=\ext^1_{\bf D}(M'', N')$.
\end{lemma}

\begin{proof}
We first consider the case where $i_{u-1}$ is a source. We have an exact sequence
$$
\begin{CD}
0 @>>> \Delta(i_{u-1}) @>>> N' @>>> X @>>> 0.
\end{CD}
$$
By applying $\hom_{\bf D}(M'', -)$, we have an exact sequence
$$
\begin{CD}
\ext^1_{\bf D}(M'', \Delta(i_{u-1})) @>>> \ext^1_{\bf D}(M'', N') @>>> \ext^1_{\bf D}(M'', X).
\end{CD}
$$
By Corollary \ref{extensionlinear}, we have $\ext^1_{\bf D}(M'', \Delta(i_{u-1}))=0$.
Note that the vertex set $\supp(X)_0$ of the support of $X$ is contained in
$\{1, ..., i_{u-1}-1\}$ and the support
of the top of the syzygy of $M''$ is contained in $\{i_{u-1}, ..., i_u, i_u+1\}$. Therefore
$\ext^1_{\bf D}(M'', X)=0$. Thus $\ext^1_{\bf D}(M'', N')=0$.

By applying $\hom_{\bf D}(M', -)$ to the exact sequence
$$
\begin{CD}
0 @>>> \Delta(i_{u-1}) @>>> N'' @>>> Y @>>> 0, 
\end{CD}
$$ 
we have an exact sequence
$$ 
\begin{CD}
\ext^1_{\bf D}(M', \Delta(i_{u-1})) @>>> \ext^1_{\bf D}(M', N'') @>>> \ext^1_{\bf D}(M', Y).
\end{CD}
$$
By Lemma \ref{prooflemma1}, we have $\ext^1_{\bf D}(M', \Delta(i_{u-1}))=0 $.
By the construction of $M'$,  we have $\hom_{\bf D}(\Omega(M'), Y)=0$, and so
$\ext^1_{\bf D}(M', Y)=0$. Hence $ \ext^1_{\bf D}(M', N'')=0$.

Now suppse that $i_{u-1}$ is a sink. By  applying $\hom_{\bf D}(-, N')$ and $\hom_{\bf D}(-, N'')$,
respectively, to the following exact sequences,
$$
\begin{CD}
0 @>>> Y @>>> M'' @>>>  \Delta(i_{u-1}) @>>> 0 \mbox{  and  }
\end{CD}
$$
$$
\begin{CD}
0 @>>> X @>>> M' @>>>  \Delta(i_{u-1})@>>> 0,
\end{CD}
$$
and by similar arguments as in the case where $i_{u-1}$ is a source, we have
$\ext^1_{\bf D}(M'', N')=0 $ and $\ext^1_{\bf D}(M', N'')=0$. This finishes the proof.
\end{proof}

\begin{lemma} \label{prooflemma3}
The following are equivalent:
\begin{itemize}
\item[(i)] $\ext^1_{\bf D}(M, N)=0$.
\item[(ii)] $\ext^1_{\bf D}(M', N')=0$, $\ext^1_{\bf D}(M'', N'')=0$, and either $M'\geq_{i_{u-1}}N'$
or $M''\geq_{i_{u-1}} N''$.

\end{itemize}
\end{lemma}

\begin{proof}
We first consider the case where $i_{u-1}$ is  a source. By applying $\hom_{\bf D}(-, N)$
to the gluing sequence of $M$ at $i_{u-1}$,
$$
\begin{CD}
0 @>>> \Delta(i_{u-1}) @>>> M'\oplus M'' @>>> M @>>> 0, 
\end{CD}
$$
we see that
$\ext^1_{\bf D}(M, N)=0$ if and only if the following hold.
\begin{itemize}
\item[(1)]$\hom_{\bf D}(M'\oplus M'', N)\rightarrow \hom_{\bf D}(\Delta(i_{u-1}), N)$ is surjective.
\item[(2)]  $\ext^1_{\bf D}(M'\oplus M'', N)=0$.
\end{itemize}
We claim that (1) holds if and only if either $M'\geq_{i_{u-1}}N'$
or $M''\geq_{i_{u-1}} N''$. Assume that (1) holds. Let $\lambda:
\Delta(i_{u-1})\rightarrow N$ be injective, and let
$\mu=(\begin{smallmatrix}\mu_1 \\ \mu_2
\end{smallmatrix}): \Delta(i_{u-1})\rightarrow M'\oplus M''$ be the
embedding in the gluing sequence of $M$.  By the assumption, there
exists a map $f=(f_1, f_2):M'\oplus M''\rightarrow N$ such that
$\lambda=f\mu=f_1\mu_1+ f_2\mu_2$. Since $N$ has a unique submodule
isomorphic to $\Delta(i_{u-1})$, we have that $\lambda$ factors
through either $\mu_1$ or $\mu_2$. If $\lambda$ factors through
$\mu_1$, then the inclusion of $\Delta(i_{u-1})$ into $N$ factors
through $f_1$,    since $\im (\mu_1 )\subseteq M'$. Hence
$M'\geq_{i_{u-1}}N'$. Similarly, $M''\geq_{i_{u-1}}N''$ if $\lambda$
factors through $\mu_2$. The converse is similar.

By Lemmas \ref{prooflemma1} and \ref{prooflemma2} and by applying
$\hom_{\bf D}(M'\oplus M'',-)$ to the gluing sequence of $N$, we get
that $\ext^1_{\bf D}(M'\oplus M'',N)=0$ if and only if $\ext^1_{\bf
D}(M',N')=0=\ext^1_{\bf D}(M'',N'')$. This proves the equivalence of
(i) and (ii) in the case where $i_{u-1}$ is a source.

Now suppose that $i_{u-1}$ is a sink. By applying $\hom_{\bf D}(M, -)$ to the gluing sequence of $N$
we see that $\ext^1_{\bf D}(M, N)=0$ if and only if the following hold.
\begin{itemize}
\item[(1)]$\hom_{\bf D}(M, N'\oplus N'')\rightarrow \hom_{\bf D}(M, \Delta(i_{u-1}))$
is surjective´.
\item[(2)] $ \ext^1_{\bf D}(M, N'\oplus N'')=0$.
\end{itemize}
Now similar to the case of a source, (1) holds if and only if either
$M'\geq_{i_{u-1}}N'$ or $M''\geq_{i_{u-1}} N''$ and (2) holds if and
only if $\ext^1_{\bf D}(M', N')=0$ and $\ext^1_{\bf D}(M'', N'')=0$.
Hence the equivalence of (i) and (ii) follows.
\end{proof}

We are now ready to prove  Lemma \ref{constructionexceptional}.

\begin{proof}[Proof of Lemma \ref{constructionexceptional}]
Suppose that $M\oplus N$ is exceptional, by Lemma \ref{gluingexceptional}, both $M'\oplus N'$ and
$M''\oplus N''$ are exceptional. By Corollary \ref{lineartotalorder}, either $M''\geq_{i_{u-1}} N''$ or
$N''\geq_{i_{u-1}}M''$.
Similarly, $M'\geq_{i_{u-1}}N'$ or $N'\geq_{i_ {u-1}}M'$ by
induction, using Lemma \ref{generalorder} and Proposition
\ref{linearsupport}. If $M''\not \geq_{i_{u-1}} N''$, then
$M'\geq_{i_{u-1}}N'$ by Lemma \ref{prooflemma3}. Thus (1) holds. If
$N'\not\geq_{i_{u-1}} M'$, then again $N''\geq_{i_{u-1}}M''$ by
Lemma \ref{prooflemma3}. Thus again (1) holds. Similarly, (2) holds
if $N''\not \geq_{i_{u-1}} M''$ or $M'\not\geq_{i_{u-1}}N'$.

We now prove the converse. By Lemma \ref{linearsourceorder} and
Lemma \ref{linearsinkorder}, we have either
$\supp_{\Delta}(M'')\subseteq \supp_{\Delta}(N'')$ or
$\supp_{\Delta}(N'')\subseteq \supp_{\Delta}(M'')$. Now by Corollary
\ref{extensionlinear} and that both $M''$ and $N''$ are
$\Delta$-filtered, we have $\ext_{\bf D}^1(M'', N'')=0$ and
$\ext_{\bf D}^1(N'', M'')=0$. Now the converse follows from Lemma
\ref{prooflemma3} and the fact that $M'\oplus N'$ is exceptional.
\end{proof}

\begin{proof}[Proof of Theorem \ref{maintheorem1}]
We will show that $M(\dv d)$ is exceptional, by induction on the number of
interior admissible vertices of $Q$.
If $Q$ has no interior admissible vertices, then by Theorem
\ref{bhrrlinear} we know that  $M(\dv d)$ is exceptional. Note that
$M(\dv d)$ is always exceptional if $\supp(\dv d)$ is linearly
oriented, following from Proposition \ref{linearsupport}. Assume
that $M(\dv d)$ is exceptional if $Q$ has less than $t$ interior
admissible vertices. Now suppose that $Q$ has $t>0$ interior
admissible vertices. We shall show that $M(\dv d)$ is exceptional.

We use the same notation as in the construction of $M(\dv d)$.
Recall that we have
$$
M(\dv c^t)=\bigoplus_l M_{K^t_l}
$$
where the $M_{K^t_l}$ are ordered such that
\begin{itemize}
\item[(1)] $(\underline{\dim}_\Delta( M_{K^t_l} ) )_{i_t}=1$ for $l=1,\dots,d_{i_t}$;
\item[(2)] $M_{K^t_l} \geq_{i_t} M_{K^t_{l+1}}$ for $l=1,\dots,d_{i_t}-1$.
\end{itemize}

By the construction of $M(\dv d)$, we may write
$$M(\dv d)
=\bigoplus_{l=1}^{d_{i_{t}}}M_{K^{t+1}_l} \oplus
\bigoplus_{l>{d_{i_{t}}}}M_{K^t_l} \oplus
\bigoplus_{l>d_{i_{t}}}M_{J^{t+1}_l}, $$ where
$\bigoplus_{l>{d_{i_{t}}}}M_{K^t_l}$ is a direct sum of the
indecomposable direct summands of $M(\dv c^t)$ whose
$\Delta$-support does not contain vertex $i_t$,
$\bigoplus_{l>d_{i_{t}}}M_{J^{t+1}_l}$ a direct sum of the
indecomposable direct summands of $M(\dv d^{t+1})$ whose
$\Delta$-support does not contain vertex $i_t$, and $M_{K^{t+1}_l}$
is obtained by gluing $M_{K^t_l}$ and $M_{J^{t+1}_{d_{i_t}+1-l}}$ at
vertex $i_t$ for $1\leq l\leq d_{i_t}$. For convenience, we denote
the representations $\bigoplus_{l=1}^{d_{i_{t}}}M_{K^{t+1}_l}$,
$\bigoplus_{l>{d_{i_{t}}}}M_{K^t_l}$ and $
\bigoplus_{l>d_{i_{t}}}M_{J^{t+1}_l}$
 by $X, Y$ and $Z$, respectively.

We have that $Y$ is exceptional by induction, and following from Proposition \ref{linearsupport},
$Z$ is exceptional. By our construction and Lemma \ref{constructionexceptional} we know that $X$ is
exceptional.
Comparing the support of
$Y$ and $Z$, we see that $\hom_{\bf D}(\Omega(Z),Y)=0=\hom_{\bf D}(\Omega(Y),Z)$, and so
$\ext^1_{\bf D}(Y, Z)=0=\ext^1_{\bf D}(Z, Y)$.
Hence $Y\oplus Z$ is exceptional.

We now prove that
$X\oplus Y$ and $Y\oplus Z$ are exceptional.
Assume that $i_{t}$ is a source.
Let $T$ be the submodule of $X$ with

\[ ( \underline{\dim}_\Delta(T) )_i=\left\{\begin{array}{ll} ( \underline{\dim}_\Delta(X) )_i & \mbox{if }
 1\leq i\leq i_{t}, \\
0 & \mbox{otherwise}. \end{array} \right.\]
We have a short exact sequence $\xymatrix{0 \ar[r]& T
\ar[r]& X \ar[r]& U \ar[r]& 0}$. Now $T\oplus Y$ is exceptional by induction, and
$\hom_{\bf D}(\Omega(Y),U)=0=\hom_{\bf D}(\Omega(U),Y)$ which shows that $U\oplus Y$ is exceptional.
Hence $X\oplus Y$ is exceptional.  Let $S$ be the submodule of $X$ with

\[ ( \underline{\dim}_\Delta(S) )_i=\left\{\begin{array}{ll} ( \underline{\dim}_\Delta(X) )_i & \mbox{if }
i_{t}\leq i\leq n,  \\
0 & \mbox{otherwise}. \end{array} \right.\]
There is a short exact sequence $\xymatrix{0 \ar[r]& S
\ar[r]& X \ar[r]&  V \ar[r]& 0}$. Now $S\oplus Z$ is exceptional by
Proposition \ref{linearsupport}, and $V\oplus Z$ is exceptional since
$\hom_{\bf D}(\Omega(Z),V)=0=\hom_{\bf D}(\Omega(V),Z)$. This proves that $X\oplus Z$ is exceptional.
The case where $i_{t}$ is a sink can be done similarly. This completes the proof
that $M(\dv d)$ is exceptional.

Let $\dv b=\sum_id_i\underline{\dim} (\Delta(i) )$. By Voigt's lemma
\cite{Voigt} we know that the  $\gl(\dv b)$-orbit of $M(\dv d)$ in
$\rep_\Delta(\tilde{Q},\mathcal{I},\dv b)$ is open, and therefore
dense, by Proposition \ref{irreducibility}. Hence $M(\dv d)$ is
unique up to isomorphism.
\end{proof}

\begin{corollary}
Let $N$ be a $\Delta$-filtered module with $\Delta$-dimension vector
$\dv d$. Then $N$ is contained in the Zariski closure of the
$\gl(\dv b)$-orbit of the exceptional $ \Delta$-filtered module
$M(\dv d)$, where $\dv b=\sum_i d_i\underline{\dim} (\Delta(i) )$.
\end{corollary}

\begin{proof}
This follows from the proof of the uniqueness of $M(\dv d)$ in Theorem \ref{maintheorem1}.
\end{proof}

Two elements $j,j'\in \{1,\dots,n\}$ are called non-comparable if
$j\not\preceq j'$ and $j'\not\preceq j$.

\begin{corollary}
The map $J\mapsto M_J$ defines a bijection between isomorphism classes of
indecomposable exceptional $\Delta$-filtered modules and subsets
$J\subseteq \{1,\dots,n\}$ satisfying the following conditions:
\begin{itemize}
\item[(1)] If $j\succeq i$ and $j'\succeq i$ for two non-comparable $j,j'\in
  J$, then $i\in J$.
\item[(2)] If $j\preceq i$ and $j'\preceq i$ for two non-comparable
  $j,j'\in J$, then $i\in J$.
\end{itemize}
\end{corollary}

\begin{proof}
By Proposition \ref{basicproperties}, we see that $M_J$ is
indecomposable if and only if $J$ satisfies the conditions (1) and
(2). That $M_J$ is exceptional follows from Theorem
\ref{maintheorem1}.
\end{proof}

\begin{example}\label{explicit}
Imitating the example in Section 8 in \cite{BHRR}, we can interpret
the exceptional $\Delta$-filtered module $M(\dv d)$ in Example \ref{arbitraygluing}
as follows.

(1) We take the $\Delta$-composition factors of $M(\dv d^i)$ as vertieces and connect
them by arrows: for
any two composition factors $\Delta(i_j)$ and $\Delta(i_l)$ of an indecomposable
direct summands of  $M(\dv d^i)$, we draw
a double arrow from $\Delta(i_j)$ to $\Delta(i_l)$, if there is a $\beta$-path from
vertex $i_j$ to vertex $i_l$ and if this indecomposable direct summand
has no composition factor $\Delta(i_s)$
such that $i_j\succ i_s\succ i_l$. In this way, we achieve the following diagrams.

$$\xymatrix{ &&&&\Delta(4)& \Delta(4)\ar@{=>}[r]&\Delta(5) \\
\Delta(1)\ar@{=>}[r]&\Delta(2)& \Delta(2)&&\ar@{=>}[ll]\Delta(4) &\Delta(4)\ar@{=>}[r]&\Delta(5)\\
&\Delta(2) & \Delta(2) &\ar@{=>}[l]\Delta(3) &\ar@{=>}[l]\Delta(4) & \Delta(4)
}$$
There are three columns, representating the three $\Delta$-filtered modules
with a linear $\Delta$-support, $M(\dv d^1)$, $M(\dv d^2)$ and $M(\dv d^3)$,
respectively.
From top to bottom: the first column is ordered from small to big according
to $>_2$. The
second column is ordered from big to small for those indecomposable direct summands of $M(\dv d^2)$
with the 2nd entry of its $\Delta$-dimension vector non-zero, according to $>_2$ and
from small to big  according to $>_4$. Finally the third column is
ordered from big to small according to $>_4$. We also make sure that the pieces to be glued
are in the same row.

(2) Now identify the $\Delta$-factors, at which the neighboring
modules are to be glued, and we achieve the module $M(\dv d)$ as
follows.

$$\xymatrix{ &&& \Delta(4)\ar@{=>}[r]&\Delta(5) \\
\Delta(1)\ar@{=>}[r]& \Delta(2)&&\ar@{=>}[ll]\Delta(4) \ar@{=>}[r]&\Delta(5)\\
& \Delta(2) &\ar@{=>}[l]\Delta(3) &\ar@{=>}[l]\Delta(4)
}$$






\end{example}

\section{Seaweed Lie algebras and Richardson elements}\label{seaweed:section}

In this section, we recall some basic defintions and results related to seaweed Lie
algebras and Richardson elements. We also give examples on seaweed Lie algebras.
We assume that $\mathfrak{g}$ is a reductive Lie algebra defined over an algebraically closed field
$\field$ of characteristic zero.

\begin{definition}
A Lie subalgebra $\mathfrak{q}$ of $\mathfrak{g}$ is called a
seaweed Lie algebra in $\mathfrak{g}$ if there exists a pair
$(\mathfrak{p},\mathfrak{p}')$ of parabolic subalgebras of
$\mathfrak{g}$ such that $\mathfrak{q}=\mathfrak{p}\cap
\mathfrak{p}'$ and $\mathfrak{p}+\mathfrak{p}'=\mathfrak{g}$. We
call such a pair of parabolic subalgebras weakly opposite.
\end{definition}

For example, take the pair consisting of a parabolic subalgebra and its opposite, or
the pair consisting of $\mathfrak{g}$ and a parabolic subalgebra. Thus the set of
seaweed Lie algebras contains all parabolic subalgebras and their Levi factors.

More generally, let us fix a Borel subalgebra $\mathfrak{b}$ of $\mathfrak{g}$ and a
Cartan subalgebra $\mathfrak{h}$ contained in $\mathfrak{b}$.
Denote by $R$ (resp. $R^{+}$, $R^{-}$ and $\Pi$) the root system (resp. the set of positive roots, the set of negative roots and the set of simple roots) relative to  $\mathfrak{h}$, $\mathfrak{b}$ and
$\mathfrak{g}$. For $\alpha \in R$, denote by $\mathfrak{g}_{\alpha}$ the root subspace relative to
$\alpha$.

For $S\subset \Pi$, set
$$
\mathbb{Z}S= \sum_{\alpha\in S} \mathbb{Z}\alpha \ , \
R_{S} = \mathbb{Z}S\cap R \ , \ R_{S}^{\pm}=R^{\pm}\cap R_{S} \hbox{ and }
\mathfrak{p}_{S}^{\pm} =
\mathfrak{h} \oplus \bigoplus_{\alpha\in R_{S}\cup R^{\pm}} \mathfrak{g}_{\alpha}
$$
the standard parabolic subalgebra and its opposite relative to $S$.

Let $\mathbf{G}$ be a connected reductive algebraic group whose Lie
algebra is $\mathfrak{g}$. The following result says that as in the
case of parabolic subalgebras, any seaweed Lie algebra is conjugate
to a ``standard'' one.

\begin{proposition}{\rm\cite{P,TYart2,TYbook}} \label{seaweedstandard}
\begin{itemize}
\item[(1)] Let $S,T\subset\Pi$. Then $(\mathfrak{p}_{S}^{-},\mathfrak{p}_{T}^{+})$ is a pair of
weakly opposite parabolic subalgebras.
In particular, $\mathfrak{q}_{S,T} = \mathfrak{p}_{S}^{-} \cap \mathfrak{p}_{T}^{+}$ is a seaweed
Lie algebra in $\mathfrak{g}$.
\item[(2)] Any seaweed Lie algebra in $\mathfrak{g}$ is conjugate by $\mathbf{G}$ to a seaweed Lie algebra
of the form $\mathfrak{q}_{S,T}$ with $S,T\subset\Pi$.
\end{itemize}
\end{proposition}

\begin{example}\label{seaweed:example}
For example, if $S=\Pi$ or $T=\Pi$, then $\mathfrak{q}_{S, T}$ is a parabolic subalgebra, while
if $S=T$, then $\mathfrak{q}_{S,T}$ is a Levi factor  of the parabolic subalgebra
$\mathfrak{q}_{\Pi, T}$.

Let us consider an example which is neither a parabolic subalgebra nor a reductive subalgebra.
Let $\mathrm{gl}_{n}(\field)$ be the set of $n$ by $n$ matrices, $(E_{ij})_{1\leq i,j\leq n}$
the standard basis of $\mathrm{gl}_{n}(\field)$ and $(E_{ij}^{*})_{1\leq i,j\leq n}$ its dual basis.
For $1\leq i \leq n$, denote $\varepsilon_{i}=E_{ii}^{*}$.

The set $\mathfrak{h}$ of diagonal matrices  is a Cartan subalgebra of $\mathrm{gl}_{n}(\field )$.
The corresponding root system is $R=\{ \varepsilon_{i}-\varepsilon_{j} ; 1\leq i \neq j\leq n\}$,
and $\mathfrak{g}_{\varepsilon_{i}-\varepsilon_{j}} = \field E_{ij}$ for $1\leq i\neq j\leq n$.
Set $\alpha_{i}=\varepsilon_{i}-\varepsilon_{i+1}$ for $1\leq i\leq n-1$. Then
$\Pi = \{ \alpha_{1},\dots ,\alpha_{n-1}\}$ is the set of simple roots for $R$ with respect to
$\mathfrak{b}$, the set of upper triangular matrices.

Let $n=9$, $S=\Pi \setminus \{\alpha_{1}, \alpha_{7}\}$ and
$T=\Pi \setminus \{ \alpha_{3},\alpha_{4}\}$. Then $\mathfrak{q}_{S,T}$ consists of matrices of the
form :
$$
\left(
\begin{array}{c|cc|c|ccc|cc}
* & 0 & 0 & 0 & 0 & 0 & 0 & 0 & 0  \\ \hline
* & * & * & * & * & * & * & 0 & 0  \\
* & * & * & * & * & * & * & 0 & 0  \\ \hline
0 & 0 & 0 & * & * & * & * & 0 & 0  \\ \hline
0 & 0 & 0 & 0 & * & * & * & 0 & 0 \\
0 & 0 & 0 & 0 & * & * & * & 0 & 0 \\
0 & 0 & 0 & 0 & * & * & * & 0 & 0 \\ \hline
0 & 0 & 0 & 0 &  * & * & * & * & *  \\
0 & 0 & 0 & 0 &  * & * & * & * & * \\
\end{array}
\right)
$$
\end{example}

Set $R^{+}_{S,T} = R_{S}^{+} \setminus R_{S\cap T}$,
$R^{-}_{S,T}=R_{T}^{-} \setminus R_{S\cap T}$ and
$R_{S,T}=R_{S,T}^{+}\cup R_{S,T}^{-}$. Thus
$
R_{S}^{+}\cup R_{T}^{-} = R_{S,T} \cup R_{S\cap T},
$
and we have
$$
\mathfrak{q}_{S,T} = \mathfrak{n}_{S,T}^{-} \oplus \mathfrak{l}_{S,T} \oplus \mathfrak{n}_{S,T}^{+}
$$
where
$\mathfrak{n}_{S,T}^{\pm} = \bigoplus_{\alpha\in R_{S,T}^{\pm}} \mathfrak{g}_{\alpha}$
and
$\mathfrak{l}_{S,T} = \mathfrak{h} \oplus \bigoplus_{\alpha\in R_{S\cap T}} \mathfrak{g}_{\alpha}
= \mathfrak{q}_{S\cap T,S\cap T}$.

Then $\mathfrak{l}_{S,T}$ is reductive in $\mathfrak{g}$ and
$\mathfrak{n}_{S,T}=\mathfrak{n}_{S,T}^{+} \oplus \mathfrak{n}_{S,T}^{-}$
is the nilpotent radical of $\mathfrak{q}_{S,T}$.

For the seaweed Lie algebra in Example \ref{seaweed:example}, $\mathfrak{l}_{S,T}$ consists of those elements with non-zero entries only on the diagonal blocks, while $\mathfrak{n}_{S,T}$ consists of those elements whose entries on the diagonal blocks are all zero, and $\mathfrak{n}_{S,T}^{+}$
(resp. $\mathfrak{n}_{S,T}^{-}$) consists of those elements in $\mathfrak{n}_{S,T}$
which are upper triangular (resp. lower triangular).

Let $\mathbf{P}_{S}^{-}$ and $\mathbf{P}_{T}^{+}$ be the parabolic
subgroups of $\mathbf{G}$ whose Lie algebras are $\mathfrak{p}_{S}^{-}$ and
$\mathfrak{p}_{T}^{+}$ respectively. Set $\mathbf{Q}_{S,T} =
\mathbf{P}_{S}^{-} \cap \mathbf{P}_{T}^{+}$. Then $\mathbf{Q}_{S,T}$
is a closed subgroup of $\mathbf{G}$ whose Lie algebra is
$\mathfrak{q}_{S,T}$, and $\mathfrak{n}_{S,T}$ is
$\mathbf{Q}_{S,T}$-stable. For example, for the seaweed Lie algebra
in Example \ref{seaweed:example}, we may take
$\mathbf{G}=\gl_{9}(\field )$, and $\mathbf{Q}_{S,T}$ consists of
elements of $\mathfrak{q}$ which are invertible.

\begin{definition}
An element $x\in\mathfrak{n}_{S,T}$ is called a Richardson element
of $\mathfrak{q}_{S,T}$ if the orbit closure
$\overline{\mathbf{Q}_{S,T}.x}=\mathfrak{n}_{S,T}$ (or equivalently,
$[x,\mathfrak{q}_{S,T}]=\mathfrak{n}_{S,T}$).
\end{definition}

We are concerned with the following question raised by Michel Duflo and Dmitri Panyushev independently :
\begin{question}\label{question}
Does $\mathfrak{q}_{S,T}$ has a Richardson element ?
\end{question}

It is well-known that if $\mathfrak{q}_{S,T}$ is a parabolic subalgebra, then
it has a Richardson element \cite{Richardson}
(see also \cite[Chapter 33]{TYbook} for a proof).

\section{Seaweed Lie algebras in $\mathrm{gl}(V)$ and weakly opposite flags}
\label{typeA:section}

Let $\mathfrak{g}=\mathrm{gl}(V)$, where $V$ is a $\field$-vector space
of dimension $n>0$, and $\mathbf{G}=\mathrm{GL}(V)$.
In this section, we first explain the relations between elements in
the nilpotent radical $\mathfrak{n}$ of a seaweed Lie algebra
$\mathfrak{q}$ and $\Delta$-filtered
$k\tilde{Q}/\mathcal{I}$-modules, where $Q$ is a quiver of type $A$
determined by $\mathfrak{q}$. We then study the connection between
$\mathbf{Q}$-orbits in $\mathfrak{n}$ and $\gl(\dv d)$-orbits in the
variety $\rep_{\Delta}(\tilde{Q}, \mathcal{I}, \dv d)$, where
$\mathbf{Q}$ is the closed subgroup of $\mathbf{G}$ associated to $\mathfrak{q}$ and $\dv
d$ is a vector determined by $\mathfrak{q}$. At the end we explain
how to find conveniently an explicit Richardson element from the
exceptional $\Delta$-filtered module $M(\dv d)$ constructed in
Section \ref{exceptional:section}.  Seaweed Lie
algebras in $\mathfrak{g}$ has a nice description via stabilizers of
certain pairs of partial flags.

Let $\Xi$ be the set of sequences of integers
$(a_{i})_{0\leq i\leq r}$ ($r$ not being fixed), verifying
$$
0=a_{0} < a_{1} < \cdots < a_{r}=n.
$$
For $\mathbf{a}=(a_{i})_{0\leq i\leq r} \in  \Xi$, let $\mathcal{F}_{\mathbf{a}}$ denote
the set of (partial) flags
$$
0=V_{0} \subset V_{1}\subset \cdots \subset V_{r}=V
$$
such that $\dim V_{i}=a_{i}$.

Recall that given a partial flag $\mathbf{F}$, its stabilizer
$\mathfrak{p}_{\mathbf{F}}$ (resp. $\mathbf{P}_{\mathbf{F}}$) in
$\mathfrak{g}$ (resp. in $\mathbf{G}$) is a parabolic subalgebra
(resp. parabolic subgroup), and this induces a bijection between the
set of all partial flags and the set of parabolic subalgebras of
$\mathfrak{g}$ (resp. the set of parabolic subgroups of
$\mathbf{G}$).

Let $B=(e_{1},\dots ,e_{n})$ be a basis for $V$.
A partial flag $\mathbf{F}=(V_{i})_{0\leq i\leq r}$ will be called $B$-upper
if $V_{i}=\mathrm{Vect}(e_{1}, \dots , e_{\dim V_{i}})$ for $1\leq i\leq r$.
Similarly, $\mathbf{F}$ is $B$-lower if $\mathbf{F}$ is
$(e_{n},\dots ,e_{1})$-upper.

A pair of partial flags $(\mathbf{F},\mathbf{F}')$ is called weakly opposite
if there exists a basis $B$ such that $\mathbf{F}$ is $B$-upper and $\mathbf{F}'$ is
$B$-lower.

It follows from Proposition \ref{seaweedstandard} that the map
$$
(\mathbf{F},\mathbf{F}')\mapsto \mathfrak{p}_{\mathbf{F}} \cap \mathfrak{p}_{\mathbf{F}'}
$$
induces a bijection between the set of pairs of weakly opposite flags and the set
of seaweed Lie algebras in $\mathfrak{g}$.

We shall associate to a pair of weakly oppposite flags a quiver of type $A$
and a dimension vector. Let us start with an example.

\subsection{A guiding example}
Take the seaweed Lie algebra in Example \ref{seaweed:example}. It corresponds to
a pair of weakly opposite flags $(\mathbf{F},\mathbf{F}') \in
\mathcal{F}_{\mathbf{a}}\times \mathcal{F}_{\mathbf{b}}$
where $\mathbf{a}=(0,3,4,9)$ and $\mathbf{b}=(0,2,8,9)$. Let $B=(e_{i})_{1\leq i\leq 9}$
be a basis of $V$
such that $\mathbf{F}$ is $B$-upper and $\mathbf{F}'$ is $B$-lower.
In this basis $B$, the elements of
$\mathfrak{q}=\mathfrak{p}_{\mathbf{F}}\cap \mathfrak{p}_{\mathbf{F}'}$
has the following matrix form
$$
\begin{array}{r}
\left. \begin{array}{l}
\\
\end{array} \right. E_{1}
\\ \hline
\left. \begin{array}{l}
\\ \\
\end{array} \right. E_{2}
\\ \hline
\left. \begin{array}{l}
\\
\end{array} \right. E_{3}
\\ \hline
\left. \begin{array}{l}
\\ \\ \\
\end{array} \right. E_{4}
\\ \hline
\left. \begin{array}{l}
\\ \\
\end{array} \right. E_{5}
\end{array}
\left(
\begin{array}{c|cc|c|ccc|cc}
* & 0 & 0 & 0 & 0 & 0 & 0 & 0 & 0  \\ \hline
* & * & * & * & * & * & * & 0 & 0  \\
* & * & * & * & * & * & * & 0 & 0  \\ \hline
0 & 0 & 0 & * & * & * & * & 0 & 0  \\ \hline
0 & 0 & 0 & 0 & * & * & * & 0 & 0 \\
0 & 0 & 0 & 0 & * & * & * & 0 & 0 \\
0 & 0 & 0 & 0 & * & * & * & 0 & 0 \\ \hline
0 & 0 & 0 & 0 &  * & * & * & * & *  \\
0 & 0 & 0 & 0 &  * & * & * & * & * \\
\end{array}
\right),
$$
where we set
$$
E_{1}=\mathrm{Vect}(e_{1}) \ , \
E_{2}=\mathrm{Vect}(e_{2},e_{3}) \ , \
E_{3}=\mathrm{Vect}(e_{4}) \ , \
E_{4}=\mathrm{Vect}(e_{5},e_{6},e_{7}) \ , \
E_{5}=\mathrm{Vect}(e_{8},e_{9}).
$$
The Levi factor $\mathfrak{l}$ of $\mathfrak{q}$ is then
$\mathrm{gl}(E_{1})\times \cdots \times \mathrm{gl}(E_{5})$.

For $1\leq i\leq 5$, we set $Y_{i}$ to be the sum of the subspaces
$E_{k}$ corresponding to non-zero blocks in the i-th block-column.
So
$$
Y_{1}=E_{1}\oplus E_{2} \ ,\ Y_{2}=E_{2} \ , \
Y_{3}=E_{2}\oplus E_{3} \ , \ Y_{4}=E_{2}\oplus E_{3}\oplus E_{4} \oplus E_{5} \ , \
Y_{5}=E_{5}.
$$
We define a quiver $Q$ with the set of vertices $Q_{0}=\{1 ,2,3,4,5\}$, and we have an arrow
from $i$ to $i\pm 1$ if $Y_{i}\subset Y_{i\pm 1}$. Thus, $Q$ is a quiver of type $A_{5}$ in the
following orientation :
$$
\xymatrix{ 1 & 2 \ar[l]_{\alpha_{1}} { } \ar[r]^{\alpha_{2}} & 3 \ar[r]^{\alpha_{3}} & 4 & 5 \ar[l]_{\alpha_{4}} }
$$

Now let $\mathfrak{n}$ denote the nilpotent radical of $\mathfrak{q}$.
Given $x\in \mathfrak{n}$, in the basis $B$, the matrix of $x$ has zero entries on the diagonal
blocks. We have by inspection that
$$
x(Y_{1}) \subset Y_{2} \ , \ x(Y_{2})=0 \ , \
x(Y_{3}) \subset Y_{2} \ , \ x(Y_{4})\subset Y_{3} \oplus Y_{5}
\ , \ x(Y_{5})= 0.
$$

We denote by $\beta_i$ the reverse arrow of $\alpha_i$ in $\tilde{Q}$.
We define the following representation $M^{x}$ of the double $\tilde{Q}$ of this quiver $Q$ :
the vector space $M^{x}_{i}$ is $Y_{i}$, $M^{x}_{\alpha_{i}}$ is just the canonical injection, while
$M^{x}_{\beta_{i}}$ is the projection onto $Y_{s(\alpha_{i})}$ with respect to the direct sum
decomposition above of the restriction of $x$ to $Y_{t(\alpha_{i})}$.
For example
$M^{x}_{\beta_{3}}= \pi_{1} \circ x|_{Y_{4}}$, where
$\pi_{1} : Y_{3} \oplus Y_{5} \rightarrow Y_{3}$ is the canonical projection with respect to
this direct sum decomposition.

It is now a straightforward verification that $M^{x}$ is a representation of $(\tilde{Q},\mathcal{I})$,
and since all the $\alpha_{i}$ are injective, it follows from Proposition \ref{HVTheorem2} that
$M^{x}$ is $\Delta$-filtered.

We have therefore a map from $\mathfrak{n}$ to $\rep_\Delta(\tilde{Q}, {\mathcal I}, \dv d)$
where $\dv d=(\dim Y_{i})_{1\leq i\leq 5}$.

\subsection{The formal construction}
We shall now describe this construction formally.

Let us fix $\mathbf{a}=(a_{i})_{0\leq i\leq r}$,
$\mathbf{b}=(b_{j})_{0\leq j\leq r'} \in \Xi$,
and $(\mathbf{F},\mathbf{F}')\in \mathcal{F}_{\mathbf{a}}\times \mathcal{F}_{\mathbf{b}}$
a pair of weakly opposite flags. Let $B=(e_{1},\dots ,e_{n})$ be a basis of $V$ such that
$\mathbf{F}$ is $B$-upper and $\mathbf{F}'$ is $B$-lower. Then
$\mathbf{F} =(V_{i})_{0\leq i\leq r}$ and $\mathbf{F}'=(V_{j}')_{0\leq j\leq r'}$
where $V_{0}=V_{0}'=\{ 0\} $ and $V_{i}=\mathrm{Vect} (e_{1},\dots ,e_{a_{i}})$ for $i>0$,
$V_{j}'=\mathrm{Vect}(e_{n-b_{j}+1},\dots ,e_{n})$ for $j>0$.

Let $0=m_{0} < m_{1} < \cdots < m_{l}=n$ be integers such that $\{m_{0} ,\dots ,m_{l}\}$
is the union of the $a_{i}$ and $n-b_{j}$, $0\leq i\leq r$, $0\leq j\leq r'$. So for each $0\leq s\leq l$,
either $\mathrm{Vect}(e_{1},\dots ,e_{m_{s}})$ is some $V_{i}$ or
$\mathrm{Vect}(e_{m_{s}+1}, \dots ,e_{n})$ is some $V_{j}'$.

For $1\leq s\leq l$, set
$$
E_{s} = \mathrm{Vect}(e_{m_{s-1}+1},\dots ,e_{m_{s}}),
\hbox{ and }
Y_{s} = V_{u(s)} \cap V_{\ell (s)}',
$$
where $u(s)=\min \{ i ; E_{s} \subseteq V_{i}\}$
and $\ell (s)=\min \{ j ; E_{s} \subseteq V_{j}'\}$.

By definition, for $1\leq s\leq l$,
we have either $E_{1}\oplus \cdots \oplus E_{s}=V_{i}$ for some $1\leq i\leq r$, or
$E_{s+1}\oplus\cdots \oplus E_{l}=V_{j}'$ for some $1\leq j\leq r'$.
Clearly, all the $V_{i}$ (resp. $V_{j}'$) are obtained in this way. So we can find
$p\leq s\leq q$ such that
$$
V_{u(s)} = E_{1}\oplus \cdots \oplus E_{q} \ , \
V_{\ell (s)}' = E_{p}\oplus \cdots \oplus E_{l} \hbox{ and }
Y_{s}=E_{p}\oplus \cdots \oplus E_{q}.
$$
In particular, $Y_{s}\supseteq E_{s}$.

\begin{lemma}\label{Ylemma}
Let $1\leq s\leq l-1$. Then we have one of the following 3 disjoint configurations:
\begin{itemize}
\item[(1)] $Y_{s}\cap Y_{s+1}=\{ 0\}$.
\item[(2)] $Y_{s}\subsetneq Y_{s+1}$.
\item[(3)] $Y_{s+1}\subsetneq Y_{s}$.
\end{itemize}
Moreover, for $1\leq s\leq l$ we have,
$$
\left\{ \begin{array}{lll}
Y_{s}=Y_{s-1}\oplus E_{s} &  \hbox{if  }\; Y_{s-1}\subseteq Y_{s} \hbox{ and } Y_{s+1}\not\subseteq Y_{s}, \\
Y_{s}=Y_{s-1}\oplus E_{s}\oplus Y_{s+1} & \hbox{if }\; Y_{s-1}\subseteq Y_{s} \hbox{ and }
Y_{s+1}\subseteq Y_{s}, \\
Y_{s}=Y_{s+1}\oplus E_{s} & \hbox{if }\; Y_{s-1}\not\subseteq Y_{s} \hbox{ and } Y_{s+1}\subseteq Y_{s}, \\
Y_{s}=E_{s} & \hbox{if } \; Y_{s-1}\not\subseteq Y_{s} \hbox{ and } Y_{s+1}\not\subseteq Y_{s} ,
\end{array}\right.
$$
where by convention, we set $Y_{0}=Y_{l+1}=\{ 0 \}$.
\end{lemma}
\begin{proof}
Observe first that by the definition of the $E_{s}$, the sequence
$(u(s))_{s=1,\dots ,l}$ (resp. $(l(s))_{s=1,\dots ,l}$) is weakly increasing
(resp. weakly decreasing), and moreover, if $u(s+1)> u(s)$ (resp. $l(s) > l(s+1)$),
then $u(s+1)=u(s)+1$ (resp. $l(s)=l(s+1)+1$).

Recall that $E_{s}=\mathrm{Vect}(e_{m_{s-1}+1}, \dots ,e_{m_{s}})$,
and that either $\mathrm{Vect}(e_{1},\dots ,e_{m_{s}})=V_{i}$ for
some $i$, or $\mathrm{Vect}(e_{m_{s}+1},\dots ,e_{n})=V_{j}'$ for
some $j$. In the first case, we have $i=u(s)$ and $u(s+1)>u(s)$,
while in the second case, we have $j=l(s)-1=l(s+1)$. Thus $(u(s) ,
l(s) ) \neq (u(s+1),l(s+1))$.

Now let $1\leq s\leq l-1$. Let $p \leq s\leq q$ be such that
$V_{u(s)}=E_{1}\oplus  \cdots \oplus E_{q}$ and
$V_{l(s)}'=E_{p} \oplus \cdots \oplus E_{l}$. So
$Y_{s} = E_{p} \oplus \cdots \oplus E_{q}$.

We have $3$ (disjoint) possibilities :

\begin{itemize}


\item[(1)] $u(s)=u(s+1)$.
Then $l(s)=l(s+1)+1$, and hence $Y_{s+1} = V_{u(s+1)} \cap V_{l(s+1)}'
\subseteq V_{u(s)} \cap V_{l(s)}' = Y_{s}$. Furthermore, since $l(s+1)<l(s)$,
$E_{s}\not\subseteq V_{l(s+1)}'$. It follows that
$V_{l(s+1)}' = E_{s+1}\oplus \cdots \oplus E_{l}$. Thus
$$
Y_{s+1} = E_{s+1}\oplus \cdots \oplus E_{q} \subsetneq E_{p}\oplus \cdots\oplus E_{q}=Y_{s}.
$$

\item[(2)] $l(s)=l(s+1)$.
Then $u(s+1)=u(s)+1$, and a similar argument shows that
$V_{u(s)}=E_{1}\oplus \cdots \oplus E_{s}$, and
$$
Y_{s}=E_{p}\oplus \cdots\oplus E_{s} \subsetneq E_{p} \oplus \cdots \oplus E_{m_{u(s+1)}} = Y_{s+1}.
$$

\item[(3)]$u(s)<u(s+1)$ and $l(s)>l(s+1)$.
Then from the above arguments, we have $V_{u(s)}=E_{1}\oplus \cdots \oplus E_{s}$, and
$V_{l(s+1)}' = E_{s+1}\oplus \cdots \oplus E_{l}$. Hence $Y_{s} \cap Y_{s+1} =\{ 0\}$.
\medskip
\end{itemize}
We have therefore proved  the first part of the lemma. The second
part follows easily from the $3$ cases above and the fact that
$E_{s}\subseteq Y_{s}$.
\end{proof}

As in the example, we define a  quiver $Q$ whose vertices are the integers
$\{1,\dots ,l\}$, and for $i=1,\dots ,l-1$, we have an arrow
$$
\xymatrix{ i \ar[r]^{\alpha_{i}} & {i+1} } \hbox{ if } Y_{i}\subset Y_{i+1} \hbox{ or }
\xymatrix{ i & {i+1} \ar[l]_{\alpha_{i}} } \hbox{ if } Y_{i}\supset Y_{i+1}.
$$
Note that $Q$ is not necessarily connected, however each of its connected components
is of type $A$.

Now let $\mathfrak{q}=\mathfrak{p}_{\mathbf{F}}\cap \mathfrak{p}_{\mathbf{F}'}$
be the seaweed Lie algebra associated to $(\mathbf{F},\mathbf{F}')$.
Then its Levi factor $\mathfrak{l}$ is $\mathrm{gl}(E_{1}) \times \cdots \times \mathrm{gl}(E_{l})$.

Let $\mathfrak{n}$ be the nilpotent radical of $\mathfrak{q}$. For
$x\in\mathfrak{n}$, we have that  $x(E_{s}) \subseteq
\bigoplus_{i\neq s} E_{i}$ for all $s=1,\dots ,l$. Since elements of
$\mathfrak{q}$ stabilizes the two flags, it follows from Lemma
\ref{Ylemma} that
$$
\left\{ \begin{array}{lll}
x(Y_{s}) \subseteq Y_{s-1} &  \hbox{if } Y_{s-1}\subseteq Y_{s} \hbox{ and } Y_{s+1}\not \subseteq Y_{s}, \\
x(Y_{s}) \subseteq Y_{s-1}\oplus Y_{s+1} & \hbox{if } Y_{s-1}\subseteq Y_{s} \hbox{ and }
Y_{s+1}\subseteq Y_{s}, \\
x(Y_{s}) \subseteq Y_{s+1} & \hbox{if } Y_{s-1}\not\subseteq Y_{s} \hbox{ and } Y_{s+1}\subseteq Y_{s}, \\
x(Y_{s})=\{0\} & \hbox{if } Y_{s-1}\not\subseteq Y_{s} \hbox{ and } Y_{s+1} \not\subseteq Y_{s}.
\end{array}\right.
$$
Clearly, any $x$ verifying these conditions is in $\mathfrak{n}$.

Again we denote by $\beta_i$ the reverse arrow of $\alpha_i$ in
$\tilde{Q}$. We define a representation $M^{x}$ of the double
$\tilde{Q}$ of the quiver $Q$ as follows : $M^{x}_{i}=Y_{i}$,
$M^{x}_{\alpha_{i}}$ is just the canonical injection, and
$M^{x}_{\beta_{i}}$ is the projection to $Y_{s(\alpha_{i})}$ with
respect to the direct sum decompositions in Lemma \ref{Ylemma} of the
restriction of $x$ to $Y_{t(\alpha_{i})}$.

We verify easily that $M^{x}\in \rep_{\Delta} (\tilde{Q},\mathcal{I},\dv d)$ where
$\dv d=(\dim Y_{i})_{1\leq i\leq l}$.

Let us denote by $\mathbf{Q}= \mathbf{P}_{\mathbf{F}}\cap
\mathbf{P}_{\mathbf{F}'}$. Then $\sigma \in \mathbf{Q}$ if and only
if $\sigma$ stabilizes both flags, or equivalently, $\sigma$
stabilizes $Y_{s}$ for $s=1,\dots, l$.

\begin{theorem}\label{correspondence}
There is a bijection between the $\mathbf{Q}$-orbits in
$\mathfrak{n}$ and the $\gl (\dv d)$-orbits in $\rep_{\Delta}
(\tilde{Q},\mathcal{I},\dv d)$.

In particular, $\mathfrak{n}$ contains an open $\mathbf{Q}$-orbit if
and only if $\rep_{\Delta} (\tilde{Q},\mathcal{I},\dv d)$ contains
an open $\gl (\dv d)$-orbit.
\end{theorem}
\begin{proof}
Without loss of generalities, we may fix the underlying vector space to be
$(Y_{i})_{1\leq i\leq l}$. Recall from Proposition \ref{irreducibility}
that $R^{\alpha}$, the set of elements $M$ in $\rep_{\Delta} (\tilde{Q},\mathcal{I},\dv d)$
such that $M_{\alpha_{i}}$ is the canonical injection for each $i$, meets every
$\gl (\dv d)$-orbit in $\rep_{\Delta} (\tilde{Q},\mathcal{I},\dv d)$.

By the construction above, we have an injective map
$$
\Phi : \mathfrak{n}\longrightarrow R^{\alpha} \ , \ x\mapsto M^{x}.
$$

(1) Let $M\in R^{\alpha}$. Recall that $V=E_{1}\oplus\cdots \oplus
E_{l}$, so any $v\in V$ is written uniquely as a sum $v=v_{1}+\cdots
+v_{l}$ where $v_{i}\in E_{i}$ for $1\leq i\leq l$.

Since $M_{\alpha_{j}}$ is the canonical injection for all $j$ and by the definition of
$\mathcal{I}$, we may define
$$
x (v_{i} ) = \left\{
\begin{array}{ll}
M_{\beta_{j}} (v_{i} ) & \hbox{if } E_{i}\subset Y_{s (\beta_{j})}, \\
0 & \hbox{otherwise.}
\end{array}
\right.
$$

Extending by linearity to $v$, we obtain an endomorphism $x$ of $V$
which clearly sends $Y_{s(\beta_{j})}$ into $Y_{t(\beta_{j})}$. So the
relations in $\mathcal{I}$ imply that $x\in\mathfrak{n}$. It is now straightforward
to check that $\Phi (x)=M$, and hence $\Phi$ is surjective.

Consequently, the map $\Phi$ is a bijection.

(2) Let $\sigma \in \gl (\dv d)$ be such that $\sigma (R^{\alpha})
\subset R^{\alpha}$. Then this is equivalent to the condition that
for any $i$, we have
$$
\xymatrix{ Y_{s(\alpha_{i})}  \ar[r]^{M_{\alpha_{i}}} \ar[d]_{\sigma_{s(\alpha_{i})}} & Y_{t(\alpha_{i})}
\ar[d]^{\sigma_{t(\alpha_{i})}} \\
Y_{s(\alpha_{i})}  \ar[r]^{M_{\alpha_{i}}} & Y_{t(\alpha_{i})}
}
$$
It follows that $\sigma_{s(\alpha_{i})}$ is the restriction of
$\sigma_{t(\alpha_{i})}$ to $Y_{s(\alpha_{i})}$. This allows us to
define a (unique) element $\tilde{\sigma}\in \mathbf{G}$ such that
$\sigma|_{Y_{j}}=\sigma_{j}$ for all $j\in Q_{0}$. Clearly,
$\tilde{\sigma}\in \mathbf{Q}$.

Conversely, given $\tau\in \mathbf{Q}$, then its restrictions to
$Y_{i}$ gives an element of $\gl (\dv d)$ which clearly stabilizes
$R^{\alpha}$. We may therefore identify $\mathbf{Q}$ with the
stabilizer of $R^{\alpha}$ in $\gl (\dv d)$. It follows that if
$M\in R^{\alpha}$, then the intersection of the $\gl(\dv d)$-orbit
of $M$ in $\rep(\tilde{Q}, \mathcal{I}, \dv d)$ with $R^{\alpha}$ is
$$
( \gl (\dv d) . M )\cap R^{\alpha} = \mathbf{Q}. M.
$$
Consequently, there is a bijective correspondence between
$\mathbf{Q}$-orbits in $R^{\alpha}$ and $\gl (\dv d)$-orbits in
$\rep_{\Delta} (\tilde{Q},\mathcal{I},\dv d)$.

(3) By Point (1) and a direct check, the map $\Phi$ is a
$\mathbf{Q}$-equivariant bijection, and so there is a one-to-one
correspondence between $\mathbf{Q}$-orbits in $\mathfrak{n}$ and
$\mathbf{Q}$-orbits in $R^{\alpha}$ which in turn, is in bijection
with $\gl (\dv d)$-orbits in $\rep_{\Delta}
(\tilde{Q},\mathcal{I},\dv d)$ by Point (2).

Finally, if $\rep_{\Delta} (\tilde{Q},\mathcal{I},\dv d)$ has an
open $\gl (\dv d)$-orbit, then $R^{\alpha}$ has an open
$\mathbf{Q}$-orbit, and so has $\mathfrak{n}$ by the above
correspondence.

Conversely, if $\mathfrak{n}$ has an open $\mathbf{Q}$-orbit, then
$R^{\alpha}$ has an open $\mathbf{Q}$-orbit, say $\Omega$. Then
$$
\overline{ \gl (\dv d) . \Omega } \supset \gl (\dv d). \overline{ \Omega } =\gl (\dv d).R^{\alpha} =
\rep_{\Delta} (\tilde{Q},\mathcal{I},\dv d).
$$
Hence $\rep_{\Delta} (\tilde{Q},\mathcal{I},\dv d)$ has an open $\gl (\dv d)$-orbit.
\end{proof}

It follows from Theorem \ref{correspondence} and
Theorem \ref{maintheorem1} that Question \ref{question}
has a positive answer for seaweed Lie algebras in $\mathrm{gl}(V)$.

\begin{corollary}
Let $\mathfrak{q}$ be a seaweed Lie algebra in $\mathrm{gl}(V)$. Then
$\mathfrak{q}$ has a Richardson element.
\end{corollary}

\begin{remark}\label{modality}
Given any integer $p$, we set $\mathfrak{n}_{p}= \{ x\in\mathfrak{n}
; \dim \mathbf{Q}.x = p\}$, and similarly, $R_{p} = \{ M\in
\rep_{\Delta} (\tilde{Q},\mathcal{I},\dv d) ; \dim \gl (\dv
d).M=p\}$. Then via the bijection in Theorem \ref{correspondence},
we obtain readily that
$$
\dim \mathfrak{n}_{p}-p = \dim R_{\dim \gl (\dv d) - \dim \mathbf{Q}
+ p} -(\dim \gl (\dv d) -\dim \mathbf{Q} + p)
$$
whenever $\mathfrak{n}_{p}$ is non-empty. It follows that the
modality of the $\mathbf{Q}$-action on $\mathfrak{n}$ is the same as
the modality of the $\gl (\dv d)$-action on $\rep_{\Delta}
(\tilde{Q},\mathcal{I},\dv d)$ (See \cite{Vinberg} for a definition
of the modality of an action).
\end{remark}

\begin{remark}\label{explicit2}
Using the procedure described in Example \ref{explicit} to construct
an explicit exceptional $\Delta$-filtered module with a given
$\Delta$-dimension vector, we obtain a procedure for constructing an
explicit Richardson element, imitating the construction given in
\cite{BHRR} for parabolic subalgebras (see also \cite{Baur}).

We first determine the quiver $Q$ and the vector $\dv e=(\dim
E_i)_{i=1,\dots ,l}$ associated to a given seaweed Lie algebra. Then
we construct the exceptional $\Delta$-filtered
$k\tilde{Q}/\mathcal{I}$-module $M(\dv e)$ and draw  an arrow
diagram associated to  $M(\dv e)$ as in Example \ref{explicit}. Then
we replace each double arrow $\xymatrix{\Delta(i)\ar@{=>}[r] &
\Delta(j)}$ by $\xymatrix{\Delta(i) & \ar[l]\Delta(j)}$. Finally we
replace column-wise the $\Delta(i)$'s by numbers $1$ to $n$,
starting on the left and from the bottom to the top. Set
$$
x= \sum_{ i \rightarrow j} E_{ij},
$$
where the sum runs over all arrows in the diagram.
Then $x$ is a Richardson element of $\mathfrak{q}$.

For example, the seaweed Lie algebra $\mathfrak{q}$ of Example
\ref{seaweed:example} gives us the quiver $Q$ and the
$\Delta$-dimension vector $\dv e$ as in Example \ref{explicit}. From
the diagram associated to $M(\dv e)$ in Example \ref{explicit}, we
obtain the following arrow diagram.
$$
\xymatrix{ &&& 7 & \ar[l] 9\\
1 & \ar[l] 3 \ar[rr] & & 6 & \ar[l] 8 \\
& 2 \ar[r] & 4 \ar[r] & 5
}
$$
Thus
$$
x=E_{31} + E_{24}+E_{45}+E_{36}+E_{86}+E_{97}=
\left(
\begin{array}{c|cc|c|ccc|cc}
0 & 0 & 0 & 0 & 0 & 0 & 0 & 0 & 0  \\ \hline
0 & 0 & 0 & 1 & 0 & 0 & 0 & 0 & 0  \\
1 & 0 & 0 & 0 & 0 & 1 & 0 & 0 & 0  \\ \hline
0 & 0 & 0 & 0 & 1 & 0 & 0 & 0 & 0  \\ \hline
0 & 0 & 0 & 0 & 0 & 0 & 0 & 0 & 0 \\
0 & 0 & 0 & 0 & 0 & 0 & 0 & 0 & 0 \\
0 & 0 & 0 & 0 & 0 & 0 & 0 & 0 & 0 \\ \hline
0 & 0 & 0 & 0 &  0 & 1 & 0 & 0 & 0  \\
0 & 0 & 0 & 0 &  0 & 0 & 1& 0 & 0 \\
\end{array}
\right)
$$
is a Richardson element of $\mathfrak{q}$.
\end{remark}

\begin{remark}
In this matrix setting, it is possible to extend with the same
proofs Theorem \ref{maintheorem1} and Theorem \ref{correspondence}
(in the case of stabilizers of weakly opposite flags) when the field
is algebraically closed of non-zero characteristic.
\end{remark}

\section{Remarks on seaweed Lie algebras of other types}\label{others:section}

In this section, we shall discuss the existence of Richardson elements in a simple
Lie algebra $\mathfrak{g}$ of type other than $A_{n}$.

Let us use the notations in Section \ref{seaweed:section}, and fix $S,T\subset \Pi$.
Recall that we have the decomposition
$$
\mathfrak{q}_{S,T} = \mathfrak{n}^{-}_{S,T} \oplus \mathfrak{l}_{S,T} \oplus
\mathfrak{n}_{S,T}^{+} \ , \
\mathfrak{n}_{S,T}=\mathfrak{n}_{S,T}^{+} \oplus \mathfrak{n}_{S,T}^{-}.
$$
We have already observed in Section \ref{seaweed:section}
that $\mathfrak{n}_{S,T}$ is the nilpotent radical of
$\mathfrak{q}_{S,T}$, and that $\mathfrak{l}_{S,T}=\mathfrak{q}_{S\cap T,S\cap T}$
is a reductive subalgebra of $\mathfrak{g}$.

We have a nice ``decomposition'' of $\mathfrak{q}_{S,T}$ as a sum of two parabolic subalgebras
of reductive subalgebras of $\mathfrak{g}$. More precisely, we have :

\begin{proposition}\label{structure}
Set $\mathfrak{p}_{S,T}^{\pm} = \mathfrak{l}_{S,T} \oplus \mathfrak{n}_{S,T}^{\pm}$.
Then
\begin{itemize}
\item[(1)] $\mathfrak{p}_{S,T}^{+}=\mathfrak{q}_{S,S\cap T}$ is a parabolic subalgebra
of the reductive Lie algebra $\mathfrak{q}_{S,S}$.
\item[(2)] $\mathfrak{p}_{S,T}^{-}=\mathfrak{q}_{S\cap T,T}$ is a parabolic subalgebra
of the reductive Lie algebra $\mathfrak{q}_{T,T}$.
\item[(3)] $\mathfrak{q}_{S,T} = \mathfrak{p}_{S,T}^{+} + \mathfrak{p}_{S,T}^{-}$
and $\mathfrak{p}_{S,T}^{+} \cap \mathfrak{p}_{S,T}^{-}=\mathfrak{l}_{S,T}$.
\item[(4)] $[\mathfrak{n}_{S,T}^{+} , \mathfrak{n}_{S,T}^{-}] = \{ 0 \}$.
\end{itemize}
\end{proposition}
\begin{proof}
This is straightforward.
\end{proof}

By Richardson's Theorem on the existence of Richardson elements in
parabolic subalgebras, $\mathfrak{p}_{S,T}^{\pm}$ has a Richardson
element in $\mathfrak{n}_{S,T}^{\pm}$. It follows easily from part
(4) of Proposition \ref{structure} and by considering projections
with respect to the direct sum decomposition $\mathfrak{n}_{S,T} =
\mathfrak{n}_{S,T}^{+}\oplus  \mathfrak{n}_{S,T}^{-}$ that if
$\mathfrak{q}_{S,T}$ has a Richardson element $x$, then there exist
Richardson elements $x_{\pm}$ of $\mathfrak{p}_{S,T}^{\pm}$ such
that $x=x_{+}+x_{-}$. Note that if we take arbitrary Richardson
elements of $\mathfrak{p}_{S,T}^{\pm}$, their sum need not be a
Richardson element of $\mathfrak{q}_{S,T}$. This explains in part
why there is a constraint in how to glue exceptional
$\Delta$-filtered modules with linear $\Delta$-support in Section
\ref{exceptional:section}.

Let us suppose that $\mathfrak{q}_{S,T}$ has a Richardson element
$x$, and let $x_{\pm}$ be Richardson elements of
$\mathfrak{p}_{S,T}^{\pm}$ such that $x=x_{+}+x_{-}$. Using the
action of the group $\mathbf{Q}_{S,T}$, we may fix one of the
elements, say $x_{+}$.

By definition, we have $[x,\mathfrak{q}_{S,T} ] = \mathfrak{n}_{S,T}$. By part (4) of Proposition
\ref{structure}, we deduce that $[x_{\pm}, \mathfrak{q}_{S,T}] \subseteq \mathfrak{n}_{S,T}^{\pm}$.
It follows that to obtain $\mathfrak{n}_{S,T}^{-}\subseteq [x,\mathfrak{q}_{S,T}]$, it is necessary that
$$
[x_{-}, \mathfrak{q}_{S,T}^{x_{+}}]=\mathfrak{n}_{S,T}^{-}
$$
where $\mathfrak{q}_{S,T}^{x_{+}} = \{ z\in \mathfrak{q}_{S,T} ; [z,x_{+}] = 0\}$.

Using the above necessary condition, we found a seaweed Lie algebra which does not
admit any Richardson element. Let $\mathfrak{g}$ be a simple Lie algebra of type $E_{8}$.
In the numbering of simple roots of \cite[Chapter 18]{TYbook}, let
$$
S =\Pi \setminus \{\alpha_{8}\} \ , \ T = \Pi \setminus \{ \alpha_{4},\alpha_{5}\}.
$$
A straightforward computation gives
$$
\dim \mathfrak{q}_{S,T}=81 \ , \ \dim \mathfrak{l}_{S,T}=22 \ , \
\dim \mathfrak{n}_{S,T}^{+} = 56 \ , \ \dim \mathfrak{n}_{S,T}^{-} = 3.
$$
Let $(x_{i})_{1\leq i\leq 56}$ be the basis of $\mathfrak{n}_{S,T}^{+}$ consisting of
root vectors ordered in a certain way given by our routine
``Seaweed'' in {\sc Gap4}. We set
$$
x_{+} = \sum_{i=1}^{56} p_{i}x_{i}
$$
where $p_{i}$ is the $i$-th prime number. The function ``LieCentralizer'' then gives
$\dim \mathfrak{q}_{S,T}^{x_{+}} = 25$. Thus
$[\mathfrak{q}_{S,T}, x_{+}]=[\mathfrak{p}_{S,T}^{+},x_{+}]=\mathfrak{n}_{S,T}^{+}$.
So $x_{+}$ is a Richardson element of $\mathfrak{p}_{S,T}^{+}$.

Next, ``BasisVectors ( Basis ( $\mathfrak{q}_{S,T}^{x_{+}}$ ) )'' provides a basis
$(b_{i})_{1\leq i\leq 25}$ of $\mathfrak{q}_{S,T}^{x_{+}}$. We then check that
$[b_{i},\mathfrak{n}_{S,T}^{-}] = 0$ if $1\leq i\leq 22$,
$[b_{23},\mathfrak{n}_{S,T}^{-}] = [b_{24},\mathfrak{n}_{S,T}^{-}]=\field z
$ for some $z\in \mathfrak{n}_{S,T}^{-}$,
and $(\mathrm{ad} b_{25})|_{\mathfrak{n}_{S,T}^{-}} =-1/2$.

Since $\mathfrak{n}_{S,T}^{-}$ is abelian, we deduce from the above that for all
$x_{-}\in \mathfrak{n}_{S,T}^{-}$, we have
$$
[x_{-},\mathfrak{q}_{S,T}^{x_{+}}] \subseteq
\mathrm{Vect} (x_{-} , z ) ,
$$
so $\dim [x_{-},\mathfrak{q}_{S,T}^{x_{+}}]  \leq 2$. In particular, $\mathfrak{q}_{S,T}$ does not
admit a Richardson element.

\begin{corollary}
There exists a seaweed Lie algebra in the simple Lie algebra of type $E_{8}$
which does not admit a Richardson element.
\end{corollary}

\begin{remarks}
(1) By taking generic elements, we have checked by using {\sc Gap4}
that if $\mathfrak{g}$ is a simple Lie algebra of rank $\leq 7$,
then any seaweed Lie algebra contained in $\mathfrak{g}$ has a
Richardson element.

(2) Using a type by type consideration and the fact that abelian
ideals of parabolic subalgebras are spherical \cite{PR}, we can show
that if the nilpotent radical of a seaweed Lie algebra is abelian,
then it has a Richardson element.
\end{remarks}

\begin{acknowledgements}
This work was done during the authors stay at Universit\"{a}t zu
K\"{o}ln and Universit\'e de Poitiers.  The authors would like thank
both institutes for their hospitality and support, in particular,
Abderrazak Bouaziz, Steffen K\"onig and Peter Littelmann. The authors would
like to thank Michel Duflo, Steffen K\"onig, Dmitri Panyushev, 
Patrice Tauvel and Changchang Xi for interesting questions and discussions.
\end{acknowledgements}

\bigskip

{\parindent=0cm
Bernt Tore JENSEN\\
Department of Mathematical Sciences, \\
Norwegian University of Science and Technology, \\
N-7034 Trondheim, \\
Norway.\\
email:  berntj@math.ntnu.no\\
\medskip

Xiuping SU \\
Mathematisches Insitut, \\
Universit\"{a}t zu K\"{o}ln, \\
Weyertal 86-90, 50931 K\"{o}ln, \\
Germany.\\
email: xsu@math.uni-koeln.de\\
\medskip

Rupert W.T. YU\\
UMR 6086 CNRS,Ê\\
D\'epartement de Math\'ematiques,\\
Universit\'e de Poitiers, \\
T\'el\'eport 2 - BP 30179, \\
Boulevard Marie et Pierre Curie,\\
86962 Futuroscope Chasseneuil cedex,\\
France.\\
email: yuyu@math.univ-poitiers.fr
}
\end{document}